\documentclass[a4paper,12pt]{article}
\usepackage[makeroom]{cancel}
\usepackage[table]{xcolor}
\usepackage{amsmath}
\usepackage{amsfonts}
\usepackage{geometry}
\usepackage[utf8]{inputenc}
\usepackage{enumerate} 
\usepackage{subfig}
\usepackage{tikz}
\usepackage{color}
\usepackage{float}
\usepackage{leftidx}
\usepackage{bigints}
\usepackage{amssymb,amsmath,dsfont,url,epsfig}
\usepackage{titlesec, blindtext, color}
\usepackage{graphicx,psfrag,epsf,times}
\definecolor{gray75}{gray}{0.75}

\newcommand{\sln}{\linespread{1}}
\newcommand*{\email}[1]{\href{mailto:#1}{\nolinkurl{#1}} } 
\titleformat{\chapter}[block]{\LARGE\bfseries\sln}{Chapter \thechapter}{11pt}{\newline\huge\bfseries}
\newtheorem{thm}{Theorem}[section]
\newtheorem{rem}{Remark}[section]
\newtheorem{defn}{Definition}[section]

\newenvironment{proof}{\paragraph{Proof:}}{\hfill$\square$}
\newtheorem{lem}{Lemma}[section]
\newtheorem{proposition}{Proposition}[section]
\newtheorem{corollary}{Corollary}[section]
\begin{document}
\title{Isometric models of the Funk disc and the  Busemann function}
\author{Ashok Kumar\footnote{E-mail: ashok241001@bhu.ac.in ; DST-CIMS, Banaras Hindu University, Varanasi-221005, India}, Hemangi Madhusudan Shah\footnote{E-mail: hemangimshah@hri.res.in ; Harish-Chandra Research Institute, A CI of Homi Bhabha National Institute, Chhatnag Road, Jhunsi, Prayagraj-211019, India.} and Bankteshwar Tiwari\footnote{E-mail: btiwari@bhu.ac.in; DST-CIMS, Banaras Hindu University, Varanasi-221005, India}}
\maketitle
\begin{abstract}
\noindent 
In this article, we find three isometric models of the Funk disc: Finsler upper half of the hyperboloid of two sheets model, the Finsler band model and the Finsler upper hemi sphere model; and we also find two new models of the Finsler-Poincar\'e disc. We  explicitly describe the geodesics in each model. Moreover, we  compute the Busemann function and consequently describe the horocycles  in the Funk and the Hilbert disc. Finally, we prove the asymptotic harmonicity of the Funk  disc. 
We also show that, the concept of asymptotic harmonicity of the Finsler manifolds {\it tacitly} depends on the measure, in {\it contrast} to the
Riemannian case.
\end{abstract}
\section{Introduction}
The Funk and the Hilbert metrics were introduced as supporting examples to solve the {\it Hilbert fourth problem}: \textit{Find all the metrics for which  line segments are geodesics}. The Funk metric on the unit disc is a well-known Finsler metric of constant flag curvature $-\frac{1}{4}$; whereas the Hilbert metric on the unit disc is the arithmetic symmetrization of its Funk metric and is of constant curvature $-1$. In the sequel, we denote by [FF] the Funk unit disc.\\
\textbf{[FF]}: The Funk metric on the unit disc can also be interpreted as Randers metric on the unit disc. It is the deformation of the well known Klein metric on the disc by a closed $1$-form (Subsection \ref{Sec 3.1}).\\

\noindent
Recently, the isometries between the Funk disc [FF], the Finsler-Poincar\'e disc [FP] and the Finsler-upper half plane [FH] has been established in (\cite{AMAK}, $\S 3$). In this article, the new isometric models of the Funk disc and the Finsler-Poincar\'e disc have been introduced. The Lorentzian metric in $\mathbb R^3$ is a well-known non positive definite Riemannian metric. However, its pullback on the upper  half of the hyperboloid of two sheets is a positive definite Riemannian metric, what we call, the hyperbolic metric on the unit disc. In this article, we construct a non positive definite Randers metric $F_L$, (see \eqref{eqn 3.12}), on the upper half space $\mathbb H^3$, whose pullback on the upper  half of the hyperboloid of two sheets is the well known Funk metric on the unit disc. Further, we construct a non positive definite Randers metric $F_+$, (see \eqref{eqn 3.5}), on the upper half space $\mathbb H^3$, whose pullback on the upper  hemi-sphere is the well known Funk metric on the unit disc.
Thus, the  Funk metric in the unit disc can be realized as the pull back of a Randers metric $F_L$, (see \eqref{eqn 3.12}), on the upper half of the hyperboloid of two sheets   [FUH-$1$], as well as  the pull back of a Randers metric $F_+$ on the upper hemi-sphere  [FUS-$2$].
We also show the isometry between the Finsler band model [FB] and the Finsler upper half plane [FU].
As the Funk disc and the Finsler-Poincar\'e  disc are isometric to each other, we further show that the pullback of the Randers metric $F_L$ on the upper  half of the hyperboloid of two sheets [FUH-$2$], as well as the pullback of the Randers metric $F_+$ on the upper  hemi-sphere can be realized as the  models of the Finsler-Poincar\'e  disc, termed as [FUS-$2$]. See the Figure 1  for the overview of all the isometric different models.\\
 We also find the geodesics in these models. In particular, the geodesics of the Band model are shown in  Figure 2. We compute the Busemann function in both the discs. Consequently, we find all the horocycles in these models.

\vspace{.3cm}
\noindent
In Section $2$, we discuss the  preliminaries required for the paper.
In Section $3$,  we introduce and explore  all the  $5$ isometric models of the Funk disc in detail as follows.
 \textbf{[\emph{FUH}-$1$]}: It is  the pullback of  the deformed  Lorentzian metric   on the upper half of the hyperboloid of two sheets ( Subsection $\ref{Sec 3.2}$).\\ 
 \textbf{[\emph{FUH}-$2$]}: It is  the another pullback of the deformed Lorentzian metric   on the upper half of the hyperboloid of two sheets ( Subsection $\ref{Sec 3.4}$).\\
  \textbf{[\emph{FB}]}: It is  the deformation of the Riemannian band model by closed $1$-form. The resulting metric is isometric to the Funk metric ( Subsection $\ref{Sec 3.5}$).\\
\textbf{[\emph{FUS}-$1$]}: It is  the pullback of the  deformation of the hyperbolic metric on $\mathbb{R}^3_+$  on  the upper hemisphere  ( Subsection $\ref{Sec 3.6}$).\\
\textbf{[\emph{FUS}-$2$]}:  It is  the pullback of the  deformation of the hyperbolic metric on $\mathbb{R}^3_+$   on  the upper hemisphere (Subsection $\ref{Sec 3.7}$).

\vspace{.3cm}
\noindent

\vspace{.3cm}
\noindent
We recall the Finsler-Poincar\'e disc   [FP] and the  Finsler-Poincar\'e upper half plane [FU] from  \cite{AMAK}, Theorem $3$.\\
\textbf{[\emph{FP}]}: The Finsler-Poincar\'e metric is a Randers metric on  the unit disc which is the deformation of the Poincar\'e metric by a closed  $1$-form (Subsection \ref{Sec 3.3}).\\
\textbf{[\emph{FU}]}: The Finsler-Poincar\'e upper half plane is a Randers deformation in the upper half plane by the hyperbolic metric in the upper half plane by a closed $1$-form (Subsection \ref{Sec 3.5}). 
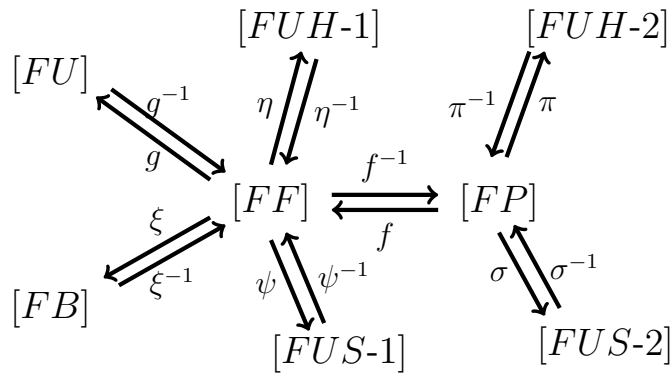
\begin{figure}[ht]
\centering
\begin{tikzpicture}
\node[ scale=1] at (-2,0)(a){\large$[FF]$};
\node[scale=1] at (1,0)(b){\large$[FP]$};
\node[scale=1] at (-1.5,2.3)(c){\large$[FUH$-$1]$};
\node[scale=1] at (-4.9,1.7)(d){\large$[FU]$};
\node[ scale=1] at (-4.9,-1.4)(e){\large$[FB]$};
\node[ scale=1] at (-1.1,-2)(f){\large$[FUS$-$1]$};
\node[ scale=1] at (2.4,-1.9)(g){\large$[FUS$-$2]$};
\node[scale=1] at (2.3,2.3)(h){\large$[FUH$-$2]$};
\draw[->,ultra thick] (-1.2,.1) -- node[above]{$f^{-1}$} (.2,.1);
\draw[->,ultra thick] (.2,-.1) -- node[below]{$f$} (-1.2,-.1);
\draw[->,ultra thick] (-2,.5) -- node[left]{$\eta$} (-1.6,2);
\draw[->,ultra thick] (-1.4,1.9) -- node[right]{$\eta^{-1}$} (-1.8,.5);
\draw[->,ultra thick] (-2.8,.3) -- node[below]{$g$} (-4.3,1.4);
\draw[->,ultra thick] (-4.1,1.5) -- node[above]{$g^{-1}$} (-2.6,.4);
\draw[->,ultra thick] (-2.8,-.2) -- node[above]{$\xi$} (-4.2,-1);
\draw[->,ultra thick] (-4,-1.1) -- node[below]{$\xi^{-1}$} (-2.6,-.3);
\draw[->,ultra thick] (-2,-.5) -- node[left]{$\psi$} (-1.5,-1.7);
\draw[->,ultra thick] (-1.3,-1.6) -- node[right]{$\psi^{-1}$} (-1.8,-.4);
\draw[->,ultra thick] (1.1,.6) -- node[right]{$\pi$} (1.6,2);
\draw[->,ultra thick] (1.4,2) -- node[left]{$\pi^{-1}$} (.9,.6);
\draw[->,ultra thick] (1,-.4) -- node[left]{$\sigma$} (1.6,-1.5);
\draw[<-,ultra thick] (1.2,-.3) -- node[right]{$\sigma^{-1}$} (1.8,-1.4);
\end{tikzpicture}
\caption{Maps in this figure are isometries obtained in Section 3.}
\end{figure}

\vspace{.3cm}
\noindent
In Section $4$, we study the geodesics of all these isometric models. We explicitly obtain the parameterization of the Funk geodesics, which are Klien lines as the point sets.
 We also show that, the geodesics of  [$FP$] are semicircles which intersect orthogonally to the boundary of the unit disc. The geodesics of  [$FU$] are  either vertical lines or semicircles centred at  $x$-axis. The geodesics of the band model are vertical lines, curves asymptotic to $x$-axis and translate of $x$-axis and some slant curves.\\
In Section $5$, we show how to find geodesics of the Hilbert metric using the Funk metric. Consequently we obtain the explicit parametrization of the Klein geodesic.\\
In Section $6$, we find the forward Busemann function for the forward ray, of the Funk disc. And in the Hilbert disc we find the Busemann function for a line. \\
In Section $7$, we show the asymptotic harmonicity of  the Funk disc with respect to the Busemann-Hausdorff volume. 

\section{Preliminaries}
The theory of Finsler manifolds can be considered as a generalization of that of Riemannian manifolds, where the Riemannian metric is replaced by a so called Finsler metric, which is a smoothly varying family of Minkowski norms in each tangent space of the manifold.\\
Let $ M $ be an $n$-dimensional smooth manifold, $T_{x}M$ denotes the tangent space of $M$ at $x$. The tangent bundle $TM$  of $M$ is the disjoint union of tangent spaces: $TM:= \sqcup _{x \in M}T_xM $. We denote the elements of $TM$ by $(x,v)$, where $v\in T_{x}M $ and $TM_0:=TM \setminus\left\lbrace 0\right\rbrace $.
 \begin{defn}[Finsler structure, \cite{SSZ}, \boldmath$\S 1.2$]\label{def 3.A01}
   A Finsler structure on the manifold $M$ is a function $F:TM \to [0,\infty)$ satisfying the following conditions:
 \\(i) $F$ is smooth on $TM_{0}$, 
 \\(ii) $F$ is a positively 1-homogeneous on the fibers of the tangent bundle $TM$,
 \\(iii) The Hessian of $\displaystyle \frac{F^2}{2}$ with elements $\displaystyle g_{ij}=\frac{1}{2}\frac{\partial ^2F^2}{\partial v^i \partial v^j}$ is positive definite on $TM_0$.\\
 The pair $(M,F)$ is called a Finsler space and $g_{ij}$ is called the fundamental tensor of the Finsler structure $F$.
 \end{defn}
 
 \noindent
 Although all Riemannian metrics are examples of Finsler metrics, however Randers metric is the simplest example of a non-Riemannian Finsler metric.
\begin{defn}[Randers Metric, \cite{SSZ}, \boldmath$\S 1.2$]
 Let $\displaystyle \alpha=\sqrt{a_{ij}(x)v^iv^j}$ be a Riemannian metric on the manifold $M$ and $\beta$ be a one-form on the manifold with $||\beta||_{\alpha}<1$, where $\displaystyle ||\beta||_{\alpha} = \sqrt {a^{ij}(x)b_{i}(x)b_{j}(x)}$, then $F(x,v)=\alpha(x,v)+\beta(x,v)$ is called a Randers metric. 
 \end{defn}

 \noindent
 It is well known that there is no canonical volume form on the Finsler manifolds, like in Riemannian case. Some of the well known volume forms on Finsler manifolds are the {\it Busemann-Hausdorff} volume form, the {\it Holmes-Thompson} volume form, the {\it maximum volume} form and the {\it minimum volume} form. A Finsler space with a volume form $d\mu$ is called a Finsler $m$-space and is denoted by $(M,F,d\mu)$.\\
 
 \noindent
 The canonical volume form in a Riemannian manifold $(M^n,\alpha), \alpha=\sqrt{a_{ij}(x)dx^i dx^j}$, is given by 
 $$dV=\sqrt{\det(a_{ij}(x))}\; dx.$$ 
 \begin{defn}[Volume forms in Finsler manifolds, \cite{SZ}, \boldmath$\S 2.2$]
    The {\it Busemann Hausdorff} volume form is defined as:
  $dV_{BH}=\sigma_{BH}(x)  dx$, where
 \begin{equation}
 \sigma_{BH}(x)=\frac{vol(B^n(1))}{vol\left\lbrace (v^i)\in T_xM: F(x,v)<1\right\rbrace}.
 \end{equation}
 The {\it Holmes-Thompson} volume form is defined as $dV_{HT}=\sigma_{HT}(x) dx$,  where 
 \begin{equation}
 \sigma_{HT}(x) =\frac{1}{vol(B^n(1))}\int_{F(x,v)<1}\det(g_{ij}(x,v)) dv.
 \end{equation}
 Here, $B^n(1)$ is the Euclidean unit ball in $\mathbb{R}^n$ and \textit{vol} is the \textit{Euclidean volume}.\\\\
 The {\it maximum} and the {\it minimum} volume form of a Finsler metric $F$ with the fundamental metric tensor $g_{ij}$ is defined as,
 \begin{equation}
 dV_{\max}=\sigma_{\max}(x)\; dx, \;\;  dV_{\min}=\sigma_{\min}(x)\; dx,
 \end{equation}
 where $ \sigma_{\max}(x)=\max\limits_{v \in I_x}   \sqrt{\det(g_{ij}(x,v))}$, 
 $\sigma_{\min}(x)=\min \limits_{v \in I_x} \sqrt{\det(g_{ij}(x,v))},$
and
 $I_x=\left\lbrace v\in T_{x}M :  F(x,v)=1 \right\rbrace $ is the indicatrix at the point $x$ of the Finsler manifold. 
 \end{defn}

\vspace{0.1in}
\noindent
 The well-known volume forms of Randers metric can be computed as given by the following lemma.
 \begin{lem}[\cite{W}, \boldmath$\S 3$ ]\label{lem 3.A1}
 The Busemann-Hausdorff volume form of a Randers metric $F = \alpha + \beta$ is given by,
 \begin{equation}\label{eqn 3.A38}
 dV_{BH} = \left( 1-||\beta||^2_\alpha\right)^{\frac{n+1}{2}} dV_\alpha,
 \end{equation}
 where $dV_\alpha=\sqrt{\det (a_{ij})} dx$.\\
 The Holmes-Thompson volume form of a Randers metric is given by,
 \begin{equation}\label{eqn 3.A39}
 dV_{HT} = dV_\alpha. 
 \end{equation}
 The maximum volume form is given by,
 \begin{equation}\label{eqn 3.A40}
 dV_{max} = \left( 1+||\beta||_\alpha\right)^{n+1} dV_\alpha.
 \end{equation}
 And the minimum volume form is given by,
 \begin{equation}\label{eqn 3.A41}
 dV_{min} = \left( 1-||\beta||_\alpha\right)^{n+1} dV_\alpha.
 \end{equation}
 \end{lem}
 
 \vspace{0.2in}
 
 \noindent
 In the Finslerian case the Hessian can not  be defined uniquely, in contrast to Riemannian case. We first need to define the gradient of a  $C^k$  $(k\geq 1)$ function on $(M,F)$.
 
 \begin{defn}[\bf{Gradient}, \cite{SZ}, \boldmath$\S 14.1$] \label{def 3.A1}
 Let $(M,F,d\mu)$ be a Finsler $m$-space,
 and $f$ be a  $C^k$  $(k\geq 1)$ function on $(M,F)$. Then the {\it gradient} $\nabla f$ is $C^{k-1}$  on $ {\cal{U}}_f:=  \left\lbrace x\in M : df_x \neq 0 \right\rbrace$ and $C^0$ on $M\setminus {\cal{U}}_f$ . Then for 
 $x \in {\cal{U}}_f$,
 \begin{equation*}\label{eqn 3.A50}
 \nabla f(x) := A^i(x, df_x) \frac{\partial}{\partial x^i},
 \end{equation*}
 where $A^i(\eta)$ are given in a standard local coordinate system $(x^i,\eta_i)$ in $T^*M$  by
  \begin{equation*}\label{eqn 3.A51}
  A^i(x, \eta)=\frac{1}{2}\frac{\partial [F^{*2}]}{\partial \eta_i}(x, \eta)=g^{*ij}(x, \eta)\eta_j, \qquad \eta \neq 0,
  \end{equation*} 
 where $F^* _x: T^*_xM \rightarrow \mathbb{R}$ is defined as $F^*_x(\xi)=\sup\limits_{F_x(v) = 1}\xi(v),  v\in T_xM$.  $F^*$ is the Finsler metric dual to $F$.
 \end{defn}
 
 
   
  \begin{defn}
      [\bf{Distance function}, \cite{SZ}, \boldmath$\S 3.2.5$  ] \label{def 3.A30}
      A locally Lipschitz function $f$ on a Finsler space $(M, F)$ is called a distance function, if 
      \begin{equation*}
          F^*(x, df_x)=1=F(x, \nabla f(x)),
      \end{equation*}
      holds almost everywhere. 
  \end{defn}

\begin{defn} 
[{\bf{Laplacian}}, \cite{SZ}, \boldmath$\S 14.1$] \label{def 3.A3}
Let $(M,F,d\mu)$ be a Finsler $m$-space
 and $f$ be a  $C^k$  $(k\geq 2)$ function on $(M,F)$. Then 
 ${div}(\nabla f)$ is a $C^{k-2}$ function on ${\cal{U}}_f$.
  Define {\it Laplacian} of $f$ as:
 \begin{equation*}\label{eqn 3.A49}
 \Delta f(x) := \mbox{div}(\nabla f(x)),\;\; x \in {\cal{U}}_f.
 \end{equation*}
 The Laplacian of $f$ is locally expressed as,
  \begin{equation}\label{eqn 3.A52}
  \Delta_{\mu } f(x)=\frac{1}{\sigma_{\mu }(x)}\frac{\partial}{\partial x^i}\left( \sigma_{\mu }(x) g^{*ij}(x, df_x)\frac{\partial f(x)}{\partial x^j}\right),
  \end{equation}
  where $\sigma_{\mu }(x)$  is the volume density of the volume form  $d\mu$.
 \end{defn}
 \begin{rem}
    The set  ${\cal{U}}_f$ is open in $M$ and $ \mu (M\setminus  {\cal{U}}_f) = 0$. 
 \end{rem}
\begin{lem}[ \cite{YZ}, Lemma 5.1]\label{lm A 2.10}
Let $(M, F, d\mu)$ be an $n$-dimensional Finsler manifold with
volume form $d\mu$ and $f$ be a differentiable function in $M$.  Then on ${\cal{U}}_f $ we have,
\begin{equation*}
    \Delta_\mu f=tr_{g_{\nabla f}} H(f)-S_\mu(\nabla f),
\end{equation*}
 where $S_\mu$, is the $S$-curvature of measure $\mu$.
\end{lem}

\vspace{.3cm}
\noindent
In what follows, we will be dealing with the
Funk and the Hilbert metric on the unit disc. 
 We first define the Funk and the Hilbert metric on a strongly convex domain $\Omega$ in $\mathbb{R}^n$.

\vspace{.3cm}
\noindent
Let $\Omega$ be a non empty strongly  convex domain in $\mathbb{R}^n$ and let  $\partial \Omega=\bar{\Omega} \setminus \Omega$ denote the  boundary of $\Omega$. For any two points $x_1, x_2; x_1 \neq x_2$ in  $\Omega$, $\overrightarrow{x_1x_2}$  denotes the  ray starting at $x_1$ and passing through $x_2$ and $\overleftarrow{x_1x_2}$ denotes the  ray starting at $x_2$ and passing through $x_1$.\\

\noindent
In the sequel, $|.|$ denotes the usual Euclidean norm in $\mathbb{R}^n$. 
 \begin{defn}
 {(\bf{Funk metric}, \cite{HHG}, Chapter 2, \boldmath$\S 2$)} The {\it Funk} metric on a strongly convex domain $\Omega$ is denoted by $d_{F,\Omega}$ and is defined as:  
 \\ For any $x_1$ and $x_2$ in $\Omega$ and  $a=\overrightarrow{x_1x_2}\cap \partial \Omega$,
 $$d_{F,\Omega}(x_1,x_2)= \log\left( \frac{|x_1-a|}{|x_2-a|}\right). $$   
 \end{defn}

 \vspace{.3cm}
 \noindent
 Note that the Funk metric is actually a weak metric in the sense that it is  not symmetric.
  It turns out that the Funk metric is actually an example of a Finsler metric. It can be shown that 
  the Funk distance $d_{F,\Omega}$ is realized by the smooth Finsler structures  $F_F$ on $ T\Omega$,
  where $ \Omega$ is a \textit{strongly convex}  domain in $\mathbb{R}^n$ with smooth boundary $\partial \Omega$ ( \cite{ohta}, $\S 1$). Thus the Finsler structure $F_F$ on $\Omega$ is given by, 
 \begin{equation}{\label{eqn 3.200}}
 F_F(x,v)=\frac{|v|}{|x-a|}, 
 \end{equation}
 where $(x,v) \in T\Omega$.

\vspace{.3cm}
\noindent
 
 \begin{defn} [\textbf{Hilbert metric}, \cite{HHG}, Chapter 3, \boldmath$\S 4$]
 The {\it Hilbert} metric on a proper convex domain $\Omega$,  denoted by $d_{H,\Omega}$,  is defined as: \\ 
 For any $x_1$ and $x_2$ in $\Omega$, let $a=\overrightarrow{x_1x_2}\cap \partial \Omega$ and $b=\overleftarrow{x_1x_2}\cap \partial \Omega$. Then
 \begin{equation}
 d_{H,\Omega}(x_1,x_2)=\frac{1}{2} \ln \left( \frac{|x_2-b|.|x_1-a|}{|x_1-b|.|x_2-a|}\right).
 \end{equation}
 \end{defn}
 It can be easily shown that, $d_{H,\Omega}$  is the distance function on $\Omega$ and satisfies the interesting property that, the line segments between any two points are minimizing. In the particular case, where $\Omega$ is the unit ball, $d_{H,\Omega}$ coincides with the  Klein metric of the hyperbolic space.
If $\partial \Omega$  is smooth and $\Omega$ is \textit{strongly convex}, then $d_H$ is realized by the smooth Finsler metric on $\Omega$, called as the Hilbert metric and is given by:
\begin{equation*}
F_H(x,v)=\frac{|v|}{2}\left\lbrace \frac{1}{|x-b|}+\frac{1}{|x-a|}\right\rbrace,
\end{equation*}
 where $(x,v) \in T\Omega$. Clearly, we have $2F_H(x,v)=F_F(x,v)+F_F(x,-v)$ (\cite{ohta}, $\S 1$).\\
 
 \noindent
The Funk metric on the unit disc in $\mathbb{R}^n$ is a special type of Randers metric of constant flag curvature, whose geodesics can be easily 
described, by the general result. 
 
\begin{thm}\label{lem 3.1}(\cite{DSSZ}, $\S 11.3$, \cite{SSZ}, $\S 3.4.8$)
  If  $F=\alpha+\beta$  is a Randers metric on a manifold $M$ with $\beta$ a closed $1$-form, then the Finslerian geodesics have the same trajectories as the geodesics of the underlying Riemannian metric $\alpha$. Moreover, if  $(M, \alpha)$ has constant curvature, then $(M, F)$ is locally projectively flat and consequently, in this case 
 $(M, F)$  is  projectively equivalent to  $(M, \alpha)$. 
 \end{thm}
\section{Isometric models of the Funk disc and the Finsler-Poincar{\'e} disc}
In this section, we introduce the isometric models of the Funk disc viz., [FUH-$1$], [FB] and [FUS-$1$] and we also introduce the isometric models of the Finsler-Poincar\'e disc viz., [FUH-$2$] and [FUS-$2$], as stated in the introduction.\\\\
 Throughout  $\langle,\rangle$ will denote the Euclidean inner product and we will be using the following notations.
\begin{itemize}
\item $\mathbb{D} = \left\lbrace(x^1, x^2)\in \mathbb{R}^2 : (x^1)^2+(x^2)^2 < 1\right\rbrace,$ the unit disc in $\mathbb{R}^2$. 
\item $\mathbb{U}=\left\lbrace(x^1, x^2)\in \mathbb{R}^2 : x^2 > 0\right\rbrace,$  the {upper half plane in}  $\mathbb{R}^2$
\item
$\mathbb{H_+} = \left\lbrace(\tilde{x}^1, \tilde{x}^2, \tilde{x}^3)\in \mathbb{R}^3 : \tilde{x}^3 = \sqrt{1+(\tilde{x}^1)^2+(\tilde{x}^2)^2} \right\rbrace,$ the
upper half of\\ the hyperboloid of two sheets in $\mathbb{R}^3$.
\item $\mathbb{B} = \left\lbrace(x^1, x^2)\in \mathbb{R}^2 :\frac{-\pi}{2}< x^2 < \frac{\pi}{2}\right\rbrace,$ 
 the {band in}  $\mathbb{R}^2$. 
\item
$\mathbb{S}^2_+=\{(x^1,x^2,x^3)\in \mathbb{R}^3 :  (x^1)^2 +(x^2)^2+(x^3)^2=1 \ \text{and} \ x^3 > 0\},$ the
upper half of the hemisphere in $\mathbb{R}^3$.
 \end{itemize}

\subsection{The Funk Disc [FF]}\label{Sec 3.1}
  \noindent
  Let us consider a proper strongly convex bounded set $\Omega \subset \mathbb{R}^n$. This will be our ground manifold. The tangent space $T_x\Omega$ at each point $x \in \Omega$ can be identified with $\mathbb{R}^n$. The Finsler structure $F_F$ on $\Omega$ is such that the unit ball centered at a point $x\in  \Omega$ is the domain $\Omega$  in the tangent space $T_x\Omega \cong \mathbb{R}^n$ itself. Thus 
 the Finsler structure $F_F$ on $\Omega$ is defined  by  \eqref{eqn 3.200}.
Let the convex set $\Omega$ be the unit disc $\mathbb{D}$,
then $$a=\left( x+\frac{v}{F_F(x,v)}\right) \in \partial \mathbb{D}  \, \Longleftrightarrow \, | x+\frac{v}{F_F(x,v)}|^2 =1.$$
 Rewriting this condition as,
$$F_F^2(x,v) (1-| x|^2)- 2F_F(x,v) \langle x , v \rangle -| v |^2=0.$$
 The non-negative root of the above quadratic equation is,  
\begin{equation}\label{eqn 3.A1}
F_F(x,v)=\alpha_F(x,v)+\beta_F(x,v),
\end{equation}
where $\displaystyle \alpha_F(x,v)= \frac{\sqrt{\left(1-|x|^2 \right) |v|^2+ \langle x ,  v \rangle ^2 } }{1-|x|^2}$ is the well known Klein metric on the unit disc and ~$\displaystyle \beta_F(x,v)=\frac{ \langle x ,  v \rangle}{1-|x|^2}$  is a $1$-form on the disc.
Since   $||\beta_F||_{\alpha_F}=|x| < 1$,  $F_F$ is a positive definite Randers metric also known as the Funk metric on the unit disc. 
Clearly, the Klein metric and the Funk metric are projectively equivalent (Theorem \ref{lem 3.1}).  

 \begin{proposition}\label{ppn A3.10}
The geodesics of the Funk metric are the line segments in the open unit disc.  
 \end{proposition}
\begin{proof} The Funk metric is given by \eqref{eqn 3.A1},
  where  $\beta_F = df_F$  with $f_F(x)=\log\frac{1}{\sqrt{1-|x|^2}}$. Thus, $\beta_F$ is exact as well as closed one form. Therefore,  by Theorem \ref{lem 3.1} the geodesics of the Funk metric and the Klein metric are pointwise same, and they are the line segments in the open unit disc. 
\end{proof}

  \subsection{The realization of the Funk metric in the unit disc on the upper sheet of the hyperboloid of two sheets \textbf{[FUH-\boldmath$1$]}}\label{Sec 3.2}
  It is well known that, the pullback of the Lorenzian metric on the upper sheet of the hyperboloid of two sheets is the realization of the Klein metric on the unit disc. In this subsection, we {\it construct} a non-positive definite Randers metric on the upper half space and show that, its pullback on the upper sheet of the hyperboloid of two sheets is the realization of the Funk metric on the unit disc. \\
  
 \noindent 
 Let $(\mathbb{R}_+^3,\alpha_L) $ denote 
 the upper half space $\mathbb{R}_+^3$ with the Lorentzian  metric $\alpha_L$ defined below, that is ,
 $$\mathbb{R}_+^3 = \left\lbrace(\tilde{x}^1,\tilde{x}^2, \tilde{x}^3)\in \mathbb{R}^3 : \tilde{x}^3 > 0\right\rbrace, $$
 $\alpha_L(\tilde{x},\tilde{v})=\sqrt{(\tilde{v}^1)^2+(\tilde{v}^2)^2-(\tilde{v}^3)^2}$ with $\tilde{x}\in \mathbb{R}_+^3 $ and $\tilde{v} \in T_{\tilde{x}}\mathbb{R}_+^3\cong \mathbb{R}^3$. \\
 Now consider the deformation $F_L$ of $\alpha_L$ by  $\displaystyle \beta_L=\frac{1}{\tilde{x}^3}d\tilde{x}^3$ in $\mathbb{R}_+^3$ as follows: 
  \begin{equation}\label{eqn 3.12}F_L(\tilde{x},\tilde{v}) =\alpha_L(\tilde{x},\tilde{v})+\beta_L(\tilde{x},\tilde{v}).\end{equation}
  It should be noted that $F_L$ is {\it a non-positive definite} Randers metric.\\
 Now  we parametrize the upper half portion $\mathbb{H_+}$
   of the hyperboloid of two sheets in $\mathbb{R}^3$ as:
\begin{equation}\label{eqn 3.A2}
\eta : \mathbb{D} \subset \mathbb{R}^2 \rightarrow \mathbb{H}_+\subset \mathbb{R}_+^3,~~~ \eta(x)=\left( \frac{x}{\sqrt{1-|x|^2}},\frac{1}{\sqrt{1-|x|^2}}\right).
\end{equation}
Note that $\eta$ is a smooth diffeomorphism between $\mathbb{D}$ and $\mathbb{H_+}$. \\

\begin{proposition}
The pullback of the metric $F_L$ defined as above, on the upper sheet of the hyperboloid of two sheets by the map $\eta$ is the realization of the Funk metric  on the upper sheet of the hyperboloid, that is,  $\eta^*F_L(x,v)= F_F(x,v)$ for all $(x,v) \in T \mathbb{D}$. 
\end{proposition}
\begin{proof}
First we find $\eta^*F_L(x,v)$, for $x\in \mathbb{D}$ and $v\in T_x \mathbb{D}\cong \mathbb{R}^2$.
 We have by \eqref{eqn 3.A2},\\
  $$\eta^1(x) = \frac{x^1}{\sqrt{1-|x|^2}} , ~ \eta^2(x) = \frac{x^2}{\sqrt{1-|x|^2}}  \; \mbox{and} \; \eta^3(x) = \frac{1}{\sqrt{1-|x|^2}}.$$
 Therefore,
\begin{equation*}
d\eta^1_x=\frac{1}{(1-|x|^2)^\frac{3}{2}}\left[ \left\lbrace 1- (x^2)^2\right\rbrace dx^1 + x^1x^2 dx^2\right],
\end{equation*}
\begin{equation*}
d\eta^2_x=\frac{1}{(1-|x|^2)^\frac{3}{2}}\left[  x^1x^2 dx^1 + \left\lbrace 1- (x^1)^2\right\rbrace dx^2\right], 
\end{equation*}
\begin{equation*}
d\eta^3_x =\frac{1}{(1-|x|^2)^\frac{3}{2}}\left[  x^1 dx^1 +x^2 dx^2\right].
\end{equation*}
Hence,
\begin{equation*}
\begin{split}
\eta^*F_L(x,v)=\left( \sqrt{(d\eta^1_x)^2+(d\eta^2_x)^2-(d\eta^3_x)^2}+\frac{1}{\eta^3(x)}d\eta^3_x\right)(v),\\ =\frac{\sqrt{\left(1-|x|^2 \right)|v|^2+\langle x , v \rangle ^2 }}{1-|x|^2}+\frac{ \langle x ,  v \rangle}{1-|x|^2} = F_F(x,v).
\end{split}
\end{equation*}
\end{proof}

\begin{rem} \label{rem 3.01}
    It is interesting to note that the Randers metric $F_L$ is not positive definite on the upper half space, however its pullback on the upper sheet of the hyperboloid of two sheets is a positive definite Randers metric, which is precisely  the Funk metric on the unit disc.
\end{rem}

\subsection{The Finsler-Poincar\'e Disc [FP]}\label{Sec 3.3}

The Poincar\'e metric on the unit disc is a model for the hyperbolic geometry in which a geodesic is represented as an arc of a circle, which intersect the disc's boundary orthogonally.  More precisely, the Poincar\'e metric $\alpha_P$ on the unit disc $\mathbb{D}$ is defined by $\displaystyle \alpha_P(x,v)=\frac{2|v|}{1-|x|^2}$, where $x=(x^1, x^2) \in \mathbb{D}$ and $v=(v^1, v^2) \in T_x\mathbb{D}$. The Finsler-Poincar\'e metric $F_P$ on the unit disc $\mathbb{D}$ is the deformation of the Poincare metric $\alpha_P$ by a one form $\beta_P$, given by $\displaystyle \beta_P(x,v)=\frac{4\langle x ,  v \rangle}{1-|x|^4} $,  $x=(x^1, x^2) \in \mathbb{D}$ and $v=(v^1, v^2) \in T_x\mathbb{D}$, and is defined as follows:
\begin{equation}\label{eqn 3.2}
F_P(x,v)=\alpha_P(x,v)+\beta_P(x,v).
\end{equation}
  Since  $\displaystyle ||\beta_P||_{\alpha_P}^2=\frac{4|x|^2}{ (1+|x|^2)^2}< 1$, the Finsler-Poincar\'e metric $F_P$ is a positive definite Randers metric (\cite{SSZ}, $\S 1.3~E$).
  \begin{rem} \label{rem 3.02}
      The Finsler-Poincar\'e metric $F_P$ is given by \eqref{eqn 3.2}, where $\beta_P = df_P,$ with $\displaystyle f_P(x)=\log\frac{1+|x|^2}{1-|x|^2}$, $x=(x^1, x^2) \in \mathbb{D}$.
Hence, $\beta_P$ is an exact as well as closed one form and therefore, in view of Theorem \ref{lem 3.1}, the Poincar\'e metric $\alpha_P$ and the Finsler-Poincar\'e metric $F_P$ are locally projectively equivalent. And the geodesics of $\alpha_P$ and $F_P$ are pointwise same.
  \end{rem}

\subsection{The realization of the Finsler-Poincar\'e disc on the upper half of the hyperboloid of two sheets \textbf{[FUH-\boldmath$2$]}}\label{Sec 3.4}

In this subsection, we show that the pullback on the upper sheet of the hyperboloid of two sheets of a non-positive definite Randers metric on the upper half space  is the realization of the Finsler-Poincar\'e metric on the unit disc. Let us consider a diffeomorphism $\pi$ between $\mathbb{D}$ and $\mathbb{H}_+$, given by:
 
\begin{equation}\label{eqn 3.A3}
\pi : \mathbb{D} \subset \mathbb{R}^2\rightarrow \mathbb{H}_+\subset \mathbb{R}_+^3, ~~~\pi(x)=\left( \frac{2x}{1-|x|^2},\frac{1+|x|^2}{1-|x|^2}\right) . 
\end{equation}

\begin{proposition}
 The pullback of the metric $F_L$ defined as above, on the upper sheet of the hyperboloid of two sheets, by the map $\pi$ is the realization of the Finsler-Poincar\'e metric on the unit disc , that is,  $\pi^*F_L(x,v)= F_P(x,v)$. 
 
 \end{proposition}
\begin{proof}
To show that $\pi^*F_L(x,v)=F_P(x,v)$. We have by \eqref{eqn 3.A3}, for $x \in \mathbb{D}$, 
  $$\pi^1(x) = \frac{2x^1}{1-|x|^2}, ~ \pi^2(x) = \frac{2x^2}{1-|x|^2}  ~ \mbox{and} ~ 
  \pi^3(x) = \frac{1+|x|^2}{1-|x|^2}. $$
 Therefore, for $v \in T_{x}\mathbb{D}$,
 \begin{equation*}
d\pi^1_x=\frac{2}{(1-|x|^2)^2}\left[ \left\lbrace 1- |x|^2+2(x^1)^2\right\rbrace dx^1 + 2x^1x^2 dx^2\right], 
\end{equation*}
\begin{equation*}
d\pi^2_x =\frac{2}{(1-|x|^2)^2}\left[  2x^1x^2 dx^1 + \left\lbrace 1- |x|^2+2(x^2)^2\right\rbrace dx^2\right], 
\end{equation*}
\begin{equation*}
d\pi^3_x=\frac{2}{(1-|x|^2)^2}\left[ 2 x^1 dx^1 +2x^2 dx^2\right]. 
\end{equation*}
Hence,
\begin{equation*}
\begin{split}
\pi^*F_L(x,v)=\left( \sqrt{(d\pi^1_x)^2+(d\pi^2_x)^2-(d\pi^3_x)^2}+\frac{1}{\pi^3(x)}d\eta^3\right)(v)\\ =\frac{2|v|}{1-|x|^2}+\frac{4\langle x ,  v \rangle}{1-|x|^4} = F_P(x,v).
\end{split}
\end{equation*}
\end{proof}

\begin{rem} \label{rem 3.03}
   We have shown that the pullback of $F_L$ on the upper sheet of the hyperboloid of two sheets $\mathbb{H}_+$ through two different 
diffeomorphisms gives the Funk as well as the  Finsler-Poincar\'e metric on the open unit disc. 
\end{rem}

\subsection{The Finsler-Poincar\'e  upper half plane \textbf{[FU]} and the Finsler Band model \textbf{[FB]}}\label{Sec 3.5}
The upper half-plane $\mathbb{U}= \left\lbrace(x^1, x^2)\in \mathbb{R}^2 : x^2 > 0\right\rbrace$ with the metric $\displaystyle \alpha_U(x,v)=\frac{|v|}{x^2}$, where $x=(x^1, x^2) \in \mathbb{U}$, $v=(v^1, v^2) \in T_x\mathbb{U}$, called the Poincar\'e upper half metric, is the standard model of two-dimensional hyperbolic geometry. The geodesics in this model of hyperbolic plane geometry are  vertical lines in $\mathbb{U}$ and upper semi circles centred on $x^1$-axis. The Finsler-Poincar\'e  upper half plane metric $F_U$ in the upper half plane is the deformation of the Poincar\'e upper half plane metric $\alpha_U$  by the $1$-form $\displaystyle \beta_U(x,v)=\frac{\langle w(x)  ,  v \rangle}{x^2(4+|x|^2)}$,  where $x=(x^1, x^2) \in \mathbb{U}$, $v=(v^1, v^2) \in T_x\mathbb{U}$ and $w(x):=(2x^1x^2,(x^2)^2-(x^1)^2-4)$; and  is defined as follows: 
\begin{equation}\label{eqn 3.3}
F_U(x,v)=\alpha_U(x,v)+\beta_U(x,v).
\end{equation}
It is easy to see that, as $\displaystyle ||\beta_U||_{\alpha_U}^2=\frac{|w(x)|}{\left( 4+|x|^2\right)^2 }< 1$, the Finsler-Poincar\'e metric $F_U$ in the upper half plane $\mathbb{U}$ is a positive definite Randers metric.\\
\begin{rem} \label{rem 3.04}
      The Finsler-Poincar\'e  upper half metric $F_U$ is given by \eqref{eqn 3.3}, where $\beta_U = df_U,$ with $\displaystyle f_U(x)=\log\frac{4+|x|^2}{x^2}$.
Hence, $\beta_U$ is an exact as well as closed one form. Therefore, by Theorem \ref{lem 3.1},  the Poincar\'e upper half metric $\alpha_U$ and the Finsler-Poincar\'e  upper half metric $F_U$ in the upper half plane $\mathbb{U}$ are locally projectively equivalent, that is, the geodesics of $\alpha_P$ and $F_P$ are pointwise same.
  \end{rem}
  
\vspace{0.1in}

\noindent
In this subsection, we find a new Finsler model corresponding to the  {\it Band model} of the Riemannian hyperbolic space, and termed it as the {\it Finsler Band model}.\\
Recall that the Band in $\mathbb{R}^2$ is: 
$$\mathbb{B} = \left\lbrace(x^1, x^2)\in \mathbb{R}^2 :\frac{-\pi}{2}< x^2 < \frac{\pi}{2}\right\rbrace.$$ And
$(\mathbb{B}, {\alpha}_B)$ is isometric to the Riemannian hyperbolic space, where for
$x \in \mathbb{B}, v \in T_{x}\mathbb{B}$, 
\begin{equation}\label{eqn3.AS}
{\alpha}_B(x, v) =\frac{|v|}{\cos x^2}.
\end{equation}
Consider one form $\beta$ on $\mathbb{B}$ defined as: 
\begin{equation}\label{form}
      \beta_{B}(x, v)=\frac{\left( e^{2x^1}-4\right) v^1 \cos x^2 +\left(e^{2x^1}+4\right) v^2 \sin x^2 }{\left(e^{2x^1}+4\right) \cos x^2}.
\end{equation}
Define a function $F_{B}$ on $T\mathbb{B}$ as: 
\begin{equation}\label{eqn 3.AA}
F_{B}(x,v) =  {\alpha}_B(x, v)  + \beta_{B}(x, v).
\end{equation}
It is easy to see that $\displaystyle ||\beta_B||_{\alpha_B}^2=\frac{\left( e^{2x^1}+4\right)^2-16 e^{2x^1}\left(\cos x^2 \right)^2  }{\left( e^{2x^1}+4\right)^2}< 1$, therefore, $F_{B}$ is a Randers metric on $\mathbb{B}$.\\
\begin{rem} \label{rem 3.05}
The Finsler metric $F_B$  is given by \eqref{eqn 3.AA}, where $\beta_B = df_B,$ with
     $$f_B(x)= x^1+ \log\left( 1+\frac{4}{e^{2x^1}}\right)+ \log \sec x^2 .$$
   So $\beta_B$ is the exact as well as  closed $1$-form. Therefore, the Finsler fundamental metric $F_B$ and $\alpha_B$ in the band $ \mathbb{B}$ are locally projectively equivalent and the geodesics of $\alpha_B$ and $F_B$ are pointwise same (Theorem \ref{lem 3.1}). 
\end{rem}

 \noindent  
Now  we show that the Finsler band model $(\mathbb{B},F_{B})$  is isometric to the
Finsler-Poincar\'e upper half plane and hence in turn isometric to the Funk disc.

 \begin{proposition}
  The Finsler band model $(\mathbb{B},F_{B})$ is isometric to the Finsler-Poincar\'e upper half plane $(\mathbb{U},F_{U})$. 
   \end{proposition}
 \begin{proof}
 Consider the map $\varphi:\mathbb{B} \rightarrow \mathbb{U}$ \; \mbox{defined by},
  $$\varphi(x)=e^{x^1}\left( -\sin x^2, \cos x^2\right).$$
  Its inverse, 
  $\varphi^{-1}:\mathbb{U} \rightarrow \mathbb{B}$ is given by,
 \begin{equation*}
 \varphi^{-1}(x)=\left( \log \sqrt{(x^1)^2+(x^2)^2},\; -\tan^{-1} \frac{x^1}{x^2}\right).
 \end{equation*}
 Clearly $\varphi$ is a diffiomorphism between $\mathbb{B}$ and $\mathbb{U}$. We claim that $\varphi$  is indeed Finslerian isometry between $(\mathbb{B},F_{B})$ and $(\mathbb{U},F_{U})$, i.e., $\varphi^*F_U(x,v)=F_B(x,v)$.\\
 By definition of the pull back $\varphi^*$, we have, $$\varphi^*F_U(x,v)=F_U\left(\varphi(x), D\varphi_x(v) \right)=\alpha_U( \varphi(x) ,  D\varphi_x(v))+ \beta _U( \varphi(x) ,  D\varphi_x(v)).$$
Since, $$ D\varphi_x(v)=-e^{x^1}( v^1 \sin x^2 + v^2  \cos x^2,    - v^1 \cos x^2  +v^2 \sin x^2),$$ for every $v =(v^1, v^2 )\in T_xD \cong \mathbb{R}^2$ and  $$w(\varphi(x))=\left(-e^{2x^1} \sin 2x^2, \; e^{2x^1} \cos 2x^2 -4\right),$$ as $w(x) =(2x^1x^2,(x^2)^2-(x^1)^2-4)$. We have, $$ \langle w(\varphi(x)) ,  D\varphi_x(v) \rangle=e^{x^1}\left( e^{2x^1}-4\right) v^1\cos x^2 +e^{x^1}\left(e^{2x^1}+4\right) v^2\sin x^2.$$
Therefore, in view of $\displaystyle \alpha_U(x,v)=\frac{|v|}{x^2}$ and $\displaystyle \beta_U(x,v)=\frac{\langle w(x)  ,  v \rangle}{x^2(4+|x|^2)}$,  we have 
   $\displaystyle \alpha _U( \varphi(x) ,  D\varphi_x(v))=\frac{|v|}{\cos x^2},$
      and
    $
      \displaystyle \beta _U( \varphi(x) ,  D\varphi_x(v))= \frac{\left( e^{2x^1}-4\right) v^1 \cos x^2 +\left(e^{2x^1}+4\right) v^2 \sin x^2 }{\left(e^{2x^1}+4\right) \cos x^2}.
      $
      Hence, 
 \begin{equation*}
 \varphi^*F_U(x,v)=F_U\left(\varphi(x), D\varphi_x(v) \right)=\alpha_U( \varphi(x) ,  D\varphi_x(v))+ \beta _U( \varphi(x) ,  D\varphi_x(v))=F_B(x,v), 
 \end{equation*}
   $\forall(x, v) \in T\mathbb{B}$. 
Thus, the map $\varphi$ is an isometry between $FB$ and $FU$.
\end{proof}

\vspace{.3cm}
\noindent
\begin{thm}
The Finsler band model $(\mathbb{B}, F_B)$  and the Funk disc model $(\mathbb{D}, F_F)$ are isometric to each other.
\end{thm}
 \begin{proof}
 Let us consider the map   $\xi:\mathbb{D} \rightarrow \mathbb{B}$, given by
   \begin{equation}\label{eqn 3.A7}
 \xi(x)=\left(\xi^1(x),\xi^2(x) \right) =\left( \log \left( 2\sqrt{\frac{1-x^1}{1+x^1}}\right) , -\tan^{-1} \left( \frac{x^2}{\sqrt{1-|x|^2}}\right) \right) ,
  \end{equation}
  with its inverse 
  \begin{equation*}
 \xi^{-1}:\mathbb{B} \rightarrow \mathbb{D},~~\xi^{-1}(x)=\left( \frac{4-e^{2x^1}}{4+e^{2x^1}},\frac{-4e^{x^1}\sin x^2}{4+e^{2x^1}}\right).
 \end{equation*}
 It suffices to show,
 \begin{equation*}
 F_F(x,v)=F_B\left(\xi(x), D\xi_x(v) \right), ~ \forall(x, v) \in T\mathbb{D},
 \end{equation*}
 where $D\xi_x$ denote the differential of $\xi$ at the point $x$.\\
      For any $v =(v^1, v^2 )\in T_x\mathbb{D} \cong \mathbb{R}^2$, the derivative $D\xi_x$ can be written as,  
       \begin{equation*}
        D\xi_x(v)=(
          -\frac{v^1}{1-(x^1)^2},
           -\frac{x^1x^2v^1}{(1-(x^1)^2)\sqrt{1-|x|^2}} - \frac{v^2}{\sqrt{1-|x|^2}}).
           \end{equation*}

\noindent
In view of \eqref{eqn3.AS} and \eqref{form},  we have, 
    \begin{equation*}
     \alpha _B( \xi(x) ,  D\xi_x(v))=\frac{\sqrt{\left(1-|x|^2 \right) |v|^2+ \langle x ,  v \rangle ^2 } }{1-|x|^2},~ ~ \beta _B( \xi(x) ,  D\xi_x(v))=\frac{\langle x ,  v \rangle}{1-|x|^2}.
    \end{equation*}

 Thus for any $(x, v) \in T\mathbb{D}$, we have,
 \begin{equation*}
 F_B\left(\xi(x), D\xi_x(v) \right)=\alpha _B( \xi(x) ,  D\xi_x(v))+ \beta _B( \varphi(x) ,  D\xi_x(v))=F_F(x,v).
 \end{equation*}
 \end{proof}

\subsection{The realization of Funk metric in the unit disc  on the upper hemi sphere \textbf{[FUS-\boldmath$1$]} }\label{Sec 3.6}
Let $\mathbb{R}_+^3 = \left\lbrace(\tilde{x}^1,\tilde{x}^2, \tilde{x}^3)\in \mathbb{R}^3 : \tilde{x}^3 > 0\right\rbrace $ be the upper half space with the hyperbolic metric $\alpha_+$, defined by $\displaystyle \alpha_+(\tilde{x},\tilde{v})=\frac{\sqrt{(\tilde{v}^1)^2+(\tilde{v}^2)^2+(\tilde{v}^3)^2}}{\tilde{x}^3}$ with $\tilde{x}=(\tilde{x}^1,\tilde{x}^2, \tilde{x}^3)\in \mathbb{R}_+^3 $ and $\tilde{v}=(\tilde {v}^1,\tilde{v}^2, \tilde{v}^3) \in T_{\tilde{x}}\mathbb{R}_+^3\cong \mathbb{R}^3$. Now let us consider the deformation of the upper half space $(\mathbb{R}_+^3,\alpha_+)$ by the one form $\displaystyle \beta_+(\tilde{x},\tilde{v})=- \frac{\tilde{v}^3}{\tilde{x}^3}$ as follows:
\begin{equation}\label{eqn 3.5}
F_+(\tilde{x},\tilde{v})=\alpha_+ (\tilde{x},\tilde{v})+ \beta_+(\tilde{x},\tilde{v}) = \frac{|\tilde{v}|}{\tilde{x}^3} - \frac{\tilde{v}^3}{\tilde{x}^3},
\end{equation}
where $\tilde{x}=(\tilde{x}^1,\tilde{x}^2, \tilde{x}^3)\in \mathbb{R}_+^3 $ and $\tilde{v}=(\tilde {v}^1,\tilde{v}^2, \tilde{v}^3) \in T_{\tilde{x}}\mathbb{R}_+^3\cong \mathbb{R}^3$.\\

\noindent
Let $\psi:\mathbb{D} \rightarrow \mathbb{S}_+^2$ be an immersion given by, 
\begin{equation}\label{eqn 3.A4}
\psi(x^1, x^2)=\left(x^1, x^2, \sqrt{1-|x|^2} \right),
\end{equation}
then it turns out that the pullback $\psi^*\alpha_+$ of the upper half hyperbolic space $(\mathbb{R}_+^3,\alpha_+)$ on  $\mathbb{D}$ is actually the Klein metric given by $\psi^*\alpha_+=\alpha_F$.

\vspace{.3cm}
\noindent
\begin{proposition}
 The pullback of the metric $F_+$ defined as above, on the upper hemi sphere by the map $\psi$ is the realization of the Funk metric on the unit disc, on the upper hemi sphere, that is,  $\psi^*F_+(x,v)= F_F(x,v) $ for all $(x, v) \in T\mathbb{D}$.
\end{proposition}
\begin{proof} In view of  \eqref{eqn 3.5} and \eqref{eqn 3.A4}, we have for all $ (x, v) \in T\mathbb{D}$,
\begin{equation*}
\psi^*F_+(x,v)=F(\psi(x), D\psi_x(v))=\frac{\sqrt{\left(1-|x|^2 \right) |v|^2+ \langle x ,  v \rangle ^2 }}{1-|x|^2}+\frac{ \langle x ,  v \rangle}{1-|x|^2}= F_F(x,v).
\end{equation*}
\end{proof}
\subsection{The realization of the Finsler-Poincar\'e  metric in the unit disc  on the upper hemi sphere \textbf{(FUS-\boldmath$2$)} }\label{Sec 3.7}
 If we consider the immersion $\sigma:\mathbb{D} \subset \mathbb{R}^2\rightarrow \mathbb{S}^2_+ \subset \mathbb{R}_+^3$ defined by,
\begin{equation}\label{eqn 3.A5}
\sigma(x^1, x^2)=\left( \frac{2x^1}{1+|x|^2},\frac{2x^2}{1+|x|^2},\frac{1-|x|^2}{1+|x|^2}\right),
\end{equation}
then the pullback $\sigma^*\alpha_+$ of the upper half hyperbolic space $(\mathbb{R}_+^3,\alpha_+)$ on  $\mathbb{D}$ is actually the well known Poincar\'e metric  $\alpha_P=\sigma^*\alpha_+$, given by \eqref{eqn 3.2}.

\vspace{.3cm}
\noindent
\begin{proposition}
 The pullback of the metric $F_+$ defined as above, on the upper hemi sphere by the map $\sigma$ is the realization of the Finsler-Poincar\'e  metric on the unit disc,  on the upper hemisphere, that is,  $\sigma^*F_+(x,v)= F_P(x,v)$ for all $(x, v) \in T\mathbb{D}$.
\end{proposition}
\textbf{Proof:}
We have by \eqref{eqn 3.A5},
\begin{equation*}
d\sigma^1_x =\frac{2}{\left(1+|x|^2\right)^2 }\left[ \left\lbrace 1+|x|^2-2(x^1)^2 \right\rbrace dx^1-2x^1x^2dx^2\right],
\end{equation*}
\begin{equation*}
d\sigma^2_x =\frac{2}{\left(1+|x|^2\right)^2 }\left[- 2x^1x^2 dx^1+\left\lbrace 1+|x|^2-2(x^2)^2\right\rbrace \right] dx^2,
\end{equation*}
\begin{equation*}
d\sigma^3_x =\frac{2}{\left(1+|x|^2\right)^2 }\left[ -2x^1 dx^1 -2x^2 dx^2\right].
\end{equation*}
Therefore,
\begin{equation*}
\sigma^*F (x, v)=\frac{2|v|}{1-|x|^2}+\frac{4\langle x ,  v \rangle}{1-|x|^4}= F_P(x,v).
\end{equation*}

\vspace{.3cm}
\noindent
Thus, we have shown  that the pullback of this hyperbolic metric on the upper hemisphere $\mathbb{S}^2_+$ through two different immersions gives the well known Funk as well  as the well known Finsler-Poincar\'e metric on the open unit disc.

\section{The Geodesics in isometric models of Funk disc}
In this section,  we describe the geodesics of all the isometric models of the Funk disc explicitly,
described in Section $3$.  In view of Section $3$, we already know the geodesics as point sets in each model, however in this section the explicit parametrization of the geodesics will be explored. We also classify all the geodesic in each model.
We first begin with the Funk metric.
By Proposition \ref{ppn A3.10}, the line segments in a convex domain $\Omega$ with the Funk metric are the Funk geodesics  as point sets. 
The explicit parametrization of the unit speed geodesics in the convex domain $\Omega$ is obtained by Athanase and Troynov in Chapter 3 of \cite{HHG}.

\begin{proposition}[\cite{HHG}, Chapter 3, \boldmath$\S 3$]
Let $\Omega \subset \mathbb{{R}}^n$ be a convex domain with the weak Finsler structure $F_f$. Then the forward unit speed linear geodesic in $\Omega$ starting at $p \in \Omega$ with  velocity $v \in T_p\Omega$ is given by,  
 \begin{equation}{\label{eqn 3.41}}
     \gamma(t)=p+\frac{(1-e^{-t})}{F_f(p, v)} v.
 \end{equation}
 \end{proposition}

\begin{proposition}
    For the Funk unit disc, the unit speed geodesic $\gamma : [0, \infty) \rightarrow \mathbb{D}$ with $\gamma(0) = p$ and the forward hitting point $y \in \partial \mathbb{D}$ is given by, 
    \begin{equation}\label{fgeo}
        \gamma(t)=e^{-t} p + (1-e^{-t})y.
    \end{equation}
    
\end{proposition}
\begin{proof}
   Consider the  particular case when $\Omega$ is open unit disc $\mathbb{D}\subset \mathbb{R}^2$. Recall that by Remark \ref{rem 3.03}, $F_f=F_F$. Let $y=(y^1, y^2) \in \partial \mathbb{D}$ be such that $ \left( p+\frac{v}{F_F(p,v)} \right)=y$. Then clearly, by \eqref{eqn 3.41} the forward unit speed Funk geodesic $\gamma : [0, \infty) \rightarrow \mathbb{D}$ with $\gamma(0) = p$ and hitting  at point $y \in \partial \mathbb{D}$  is given by (\ref{fgeo}).
   \end{proof}
   
\vspace{0.2in}
\noindent
Clearly, if $p=(p^1,p^2) \in \mathbb{D}$, $v=(v^1,v^2) \in T_p\mathbb{D}$,  then \eqref{fgeo} yields that, 
\begin{equation}\label{fgeo1}
    \gamma(t) =(\gamma_1(t), \gamma_2(t))=( e^{-t} (p^1-y^1)+y^1, e^{-t} (p^2-y^2)+y^2).
\end{equation}

\vspace{.2cm}
\noindent
Now using isometries obtained in the previous section, we classify the geodesics in various Funk model spaces.
\subsection{Geodesics in the Finsler-Poincar\'e  upper half plane \textbf{[FU]}}  
As pointed out in Remark  \ref{rem 3.04},  the geodesics of the Finsler-Poincar\'e upper half plane are pointwise same as  the Poincar\'e metric on the upper half plane. These geodesics are completely classified.
\begin{thm}
The geodesics of the Finsler-Poincar\'e  upper half plane are the vertical rays in the Finsler-Poincar\'e upper half plane as well as the semicircles centred on $x$-axis. More precisely,
\begin{enumerate}
\item[(i)] the line segments in the Funk disc passing through $(-1,0)$ correspond to the vertical rays in the upper half plane, 
\item[(ii)] the vertical line segments in the Funk disc correspond to the concentric semi circles centred at origin,
\item[] and 
\item[(iii)] the other line segments in the Funk disc correspond to the semicircles centred on $x$-axis.
\end{enumerate}
\end{thm}
\begin{proof}
Let $g:\mathbb{D} \rightarrow \mathbb{U}$ be the isometry between  $(\mathbb{D}, F_F)$ and $(\mathbb{U}, F_U)$ given by (see  \cite{AMAK}, $\S 4$, Theorem 2),
\begin{equation*}
g(x)=\left( \frac{2x^2}{1+x^1}, \frac{2\sqrt{1-|x|^2}}{1+x^1}\right).
\end{equation*}
The geodesics in $(\mathbb{U}, F_U)$ are  $g$-isometric images of the geodesics $\gamma(t)=(\gamma_1(t), \gamma_2(t))$ given by \eqref{fgeo1}. Let $\omega(t)=g(\gamma(t))=(X^1(t),X^2(t))$. Then
\begin{equation*}
\omega(t)=\left( \frac{2\gamma_2(t)}{1+\gamma_1(t)}, \frac{2\sqrt{1-|\gamma(t)|^2}}{1+\gamma_1(t)}\right).
\end{equation*} 
Therefore, $\displaystyle X^1(t)=\frac{2\gamma_2(t)}{1+\gamma_1(t)}$ and $\displaystyle X^2(t)=\frac{2\sqrt{1-|\gamma(t)|^2}}{1+\gamma_1(t)}.$ It is easy to see that, 
\begin{equation}\label{2}
(X^1(t))^2+(X^2(t))^2=4\frac{1-\gamma_1(t)}{1+\gamma_1(t)}.
\end{equation}
For any line segments (non-vertical) in the Funk disc, we have $\gamma_2(t)=m \gamma_1(t)+c$, where $m$ is the slope of the line segment and $c$ is its y-intercept.
\begin{itemize}
 \item[(i)] For a line segment in the Funk disc passing through $(-1,0)$, $\gamma_2(t)$ and $1+\gamma_1(t)$ are in constant ratio and hence $m=c$. Therefore, $X^1(t)=2c~ (\text{constant})$. Hence, such line segments in the Funk disc correspond to the vertical rays in $\mathbb{U}$ through the isometry $g$.
 \item[(ii)] For the  vertical line segments in the Funk disc, $\gamma_1(t)=k$ (constant) with $|k|< 1$,  then by \eqref{2}, 
 \begin{equation*}
 (X^1(t))^2+(X^2(t))^2=c,~ \mbox{where}~ c=4\frac{1-k}{1+k}.
 \end{equation*}
This shows that the vertical line segments in the Funk disc correspond to the concentric semicircles centred at origin in the upper half plane.
  \item[(iii)] For the line segments in the Funk disc not passing through point $(-1,0)$, we have $m\neq c$, then  by \eqref{2},
 $$\left( X^1(t)+\frac{2}{m-c}\right)^2+(X^2(t))^2 =4 \frac{(m^2-c^2+1)}{(m-c)^2}.$$
  Hence, these line segments in the Funk disc correspond to the semicircles centred on $x$-axis in the upper half plane.
  \end{itemize}
\end{proof}

\subsection{Geodesics in the Finsler-Poincar\'e disc [FP]}
\begin{thm}
The geodesics of the Finsler-Poincar\'e disc $(\mathbb{D}, F_P)$  are diametric line segments  and  the circular 
arcs that intersect orthogonally to the boundary circle (in the Euclidean sense). More explicitly,
\begin{enumerate}
\item[(i)] the line segments in the Funk disc passing through 
the centre of the disc correspond to the diametric line segments in the Finsler-Poincar\'e disc passing through the centre of the disc,

\item[] and
\item[(ii)] the other line segments in the Funk disc correspond to the  circular arcs  that intersect orthogonally to the boundary circle.
\end{enumerate}
\end{thm}
\begin{proof} Let $f:(\mathbb{D},F_F) \rightarrow (\mathbb{D},F_P),$ be the isometry between  $(\mathbb{D}, F_F)$ and $(\mathbb{D}, F_P)$, given by (see \cite{AMAK}, $\S 4$ , Theorem 1),

\begin{equation*}
 f(x)=\frac{x}{1+\sqrt{1-|x|^2}}.
\end{equation*}
Thus, the geodesics in $(\mathbb{D}, F_P)$ are $\omega(t)=f(\gamma(t))=(X^1(t),X^2(t))$, where $\gamma(t)=(\gamma_1(t), \gamma_2(t))$ is given by \eqref{fgeo1}. Hence, we obtain,
\begin{equation*}
\omega(t)=\frac{\gamma(t)}{1+\sqrt{1-|\gamma(t)|^2}}=\left( \frac{\gamma_1(t)}{1+\sqrt{1-|\gamma(t)|^2}},\frac{\gamma_2(t)}{1+\sqrt{1-|\gamma(t)|^2}}\right).
\end{equation*}
Then, 
\begin{equation*}
X^1(t)=\frac{\gamma_1(t)}{1+\sqrt{1-|\gamma(t)|^2}} \qquad \mbox{and} \qquad X^2(t)=\frac{\gamma_2(t)}{1+\sqrt{1-|\gamma(t)|^2}}.
\end{equation*}

\vspace{.3cm}
\noindent
The line segments $\gamma(t)$, represented as $\gamma(t) =(\gamma_1(t),\gamma_2(t))$, in the Funk disc can be written as $\gamma_2(t)=m \gamma_1(t)+c$. 
\begin{itemize}
\item[(i)] If $c=0$, then $\gamma_2(t)=m \gamma_1(t)$. In this case, $X^2(t)=m X^1(t).$ That is the line segments (geodesics) in the Funk disc passing through the centre of the disc correspond to the line segments in the Finsler-Poincar\'e disc passing through the centre of the disc.
\item[(ii)] If $\gamma_2(t)=m \gamma_1(t)+c$, for some real numbers $m, c$ with $c\neq 0$, then we obtain,
\begin{equation}\label{eqnn 3.4.21}
  \left(X^1(t)+ \frac{m}{c}\right)^2 +\left(X^2(t)-\frac{1}{c} \right)^2 =\left(\frac{1+m^2-c^2}{c^2} \right).  
\end{equation} 
 \item[(iii)] If $\gamma_1(t)=0$, then $X^1(t)=0$. Also, if $\gamma_1(t)=c \neq 0, |c| < 1$, then 
 \begin{equation}\label{eqnn 3.4.22}
     \left(X^1(t)- \frac{1}{c}\right)^2 +\left(X^2(t)\right)^2 =\left(\frac{1-c^2}{c^2} \right).
 \end{equation}
\end{itemize}
Clearly, the arcs of the circles given by \eqref{eqnn 3.4.21} and \eqref{eqnn 3.4.22} intersect orthogonally to the boundary (in the Euclidean sense) of the unit disc.
\end{proof}

\subsection{Geodesics in the Finsler-band model [FB]}
\begin{thm}
The geodesic $\omega(t)=\left( X^1(t),X^2(t) \right)$, 
of the Finsler-band model $(\mathbb{B}, F_B)$ satisfy,
\begin{equation}\label{eqnn 3.1}
    4e^{X^1(t)}\sin X^2(t)=\pm \Big[m(4-e^{2X^1(t)})+c(4+e^{2X^1(t)})\Big] ,
\end{equation}
\noindent
where  $m$ and $c$ are arbitrary real numbers.
\end{thm}
\begin{proof} The Funk-disc model $(\mathbb{D}, F_F)$ and the Finsler-band model $(\mathbb{B}, F_B)$ are isometric to each other via the map $\xi:\mathbb{D} \rightarrow \mathbb{B}$ is given by \eqref{eqn 3.A7}. Therefore, if $\gamma(t)=(\gamma_1(t), \gamma_2(t))$ given by \eqref{fgeo1}, we obtain, 
\begin{equation*}
 \xi(\gamma(t)) =\left( \log \left( 2\sqrt{\frac{1-\gamma_1(t)}{1+\gamma_1(t)}}\right) , -\tan^{-1} \left( \frac{\gamma_2(t)}{\sqrt{1-|\gamma(t)|^2}}\right) \right).
\end{equation*} 
Let $\omega(t)=\xi(\gamma(t))=(X^1(t),X^2(t))$, then
\begin{equation}\label{eq:band}
 X^1(t)= \log \left( 2\sqrt{\frac{1-\gamma_1(t)}{1+\gamma_1(t)}}\right) \qquad \mbox{and} \qquad X^2(t)= -\tan^{-1} \left( \frac{\gamma_2(t)}{\sqrt{1-|\gamma(t)|^2}}\right).
\end{equation}
Clearly, if $\gamma_1(t)$ is constant, then $X^1(t)$ is constant, that is, the isometric images of the vertical line segments in the Funk disc are the vertical line segments in the band model. If $\gamma$ is a non-vertical line segment in $\mathbb{D}$, then its coordinates can be written as
 $\gamma_2(t)=m \gamma_1(t)+c$. Using this expression in (\ref{eq:band}), the coordinates $(X^1(t),X^2(t))$ of the geodesics in the band model satisfy \eqref{eqnn 3.1}.
\end{proof}\\

\noindent
Some of the curves given by above equation are  drawn in the Figure $2$ (this figure is drawn with the help of Python).

\begin{figure}[ht]
\begin{center}
 \includegraphics[height=5cm,width=10.0cm]{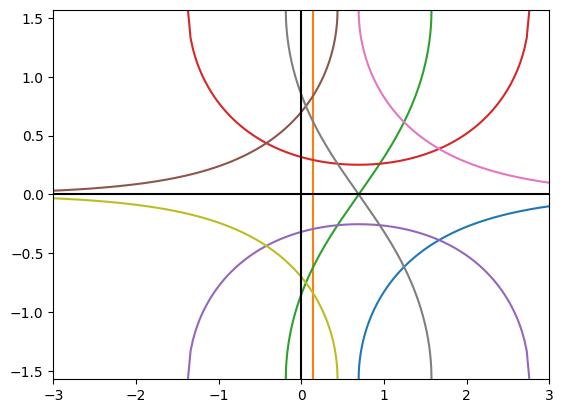}
 \end{center}
 \caption{Some geodesics in the band model.}
 \end{figure}

\section{The Geodesics of the Hilbert disc are the geodesics of the Beltrami Klein model} Note that the Hilbert metric on the unit disc  $\mathbb{D}$ is the well known Riemannian Beltrami-Klein metric.  In this section,
we find the parameterization of the Klein geodesics completely.

\vspace{.3cm}
\noindent
As stated in Section $2$, the Hilbert metric [\cite{HHG}, chapter 3, $\S 4$] is the arithmetic symmetrization of the Funk metric. Thus, for $x \in \mathbb{D}$ and $v \in T_x\mathbb{D}$,
 \begin{equation}
     F_H(x,v)=\frac{1}{2}\lbrace F_F(x,v)+F_F(x,-v)\rbrace=\frac{\sqrt{\left(1-|x|^2 \right) |v|^2+ \langle x ,  v \rangle ^2 } }{1-|x|^2}.
 \end{equation}
 The unit speed geodesic in a convex domain $\Omega$ with  Hilbert metric is explicitly described by Athanase and Troynov [\cite{HHG}, chapter 3, $\S 4$]. 
 \begin{proposition}[\cite{HHG}, Chapter 3, \boldmath$\S 4$]{\label{ppn 3.5.1}} The unit speed linear geodesic of the Hilbert metric starting at $p\in \Omega$ with  velocity $v \in T_p\Omega$ is the path, 
 \begin{equation}{\label{eqn 3.421}}
\beta(t)=p+\varphi(t)v, ~\mbox{where}~~   
\varphi(t)=\frac{e^t-e^{-t}}{F_F(p,v)e^t+F_F(p,-v)e^{-t}}.
 \end{equation}
\end{proposition}

\noindent

\begin{proposition}\label{lem 3.A0}
Let $\beta$ be the unit speed geodesic (line segment) in the Hilbert disc $\mathbb{D}$ with $\beta(0)=p$ and suppose that $\beta$ hits at $y \in \partial \mathbb{D}$ in the forward direction. Then $\beta$  is given by,
 \begin{equation}\label{eqn 3.4.21}
 \beta(t)=(1-s(t))p+s(t) y,
 \end{equation}
 where  $\displaystyle s(t)=\frac{e^t-e^{-t}}{e^t+k(p,y) e^{-t}}$\; and  \; $\displaystyle k(p,y)=\frac{|y-p|^2}{|y|^2-|p|^2}$.  
\end{proposition}
\begin{proof} For  $p=(p^1, p^2)\in \mathbb{D}$ and $v  =(v^1,v^2) \in T_p\mathbb{D}$, the unit speed  Hilbert geodesic in  $\mathbb{D}$ is
(Proposition \ref{ppn 3.5.1}),
\begin{equation}\label{eqn 3.A8}
 \beta(t)=p+\frac{e^t-e^{-t}}{F_F(p,v)e^t+F_F(p,-v)e^{-t}}v.
 \end{equation}
 In order to simplify the second term on R.H.S. of  \eqref{eqn 3.A8}, we proceed as follows:\\
 Let $y'\in \partial\mathbb{D}$ be the unique hitting point of $\beta(t)$, different from $y=(y^1,y^2)\in \partial\mathbb{D}$. \\
 It is easy to see that the coordinates of $y'=(y'^1,y'^2)$ are given by:
 	$$(y'^1,y'^2)=\left( y^1+ \frac{2(1-\langle p , y\rangle )}{1+|p|^2-2 \langle p , y \rangle }(p^1-y^1),y^2+\frac{2(1-\langle p , y\rangle )}{1+|p|^2-2 \langle p , y \rangle }(p^2-y^2)\right).$$ 
 Now we find a relation between $F_F(p,v)$ and $F_F(p,-v)$.  \\
 	As $p+\frac{v}{F_F(p,v)}=y$, we get
 	\begin{equation}\label{eqn 3.A9}
 	p^1+\frac{v^1}{F_F(p,v)}=y^1 ~~ \mbox{and}~~ p^2+\frac{v^2}{F_F(p,v)}=y^2.
 	\end{equation}
 	Also we have, $p-\frac{v}{F_F(p,-v)}=y'$ which gives,
 	\begin{equation}\label{eqn 3.A10}
 	\begin{split}
 	p^1-\frac{v^1}{F_F(p,-v)}=y^1+ \frac{2(1-\langle p , y\rangle )}{1+|p|^2-2 \langle p , y \rangle }(p^1-y^1), \\~~ \mbox{and}~~ p^2-\frac{v^2}{F_F(p,-v)}= y^2+\frac{2(1-\langle p , y\rangle )}{1+|p|^2-2 \langle p , y \rangle }(p^2-y^2).
 	\end{split}
 	\end{equation}
 	From \eqref{eqn 3.A9} and \eqref{eqn 3.A10}, we have
 	\begin{equation*}
 	\frac{1-|p|^2}{1+|p|^2-2 \langle p , y \rangle }=\frac{F_F(p,v)}{F_F(p,-v)}.
 	\end{equation*}
 	Thus,
 	\begin{equation}\label{eqn 3.A12}
 	F_F(p,-v)=k(p, y) F_F(p,v), 
 	\end{equation}
 	where 
 $\displaystyle k(p,y)=\frac{1+|p|^2-2 \langle p , y \rangle }{1-|p|^2}=\frac{|y-p|^2}{|y|^2-|p|^2}.$\\
 \noindent	
Using \eqref{eqn 3.A12} in \eqref{eqn 3.A8} yields,
 \begin{equation}
 \beta(t)=\frac{(1+k(p,y))e^{-t}}{e^t+k(p,y) e^{-t}}p+\frac{e^t-e^{-t}}{e^t+k(p,y) e^{-t}}y.
 \end{equation}
 Thus, the unit speed geodesic $\beta(t)$ with $\beta(0)=p$ and hitting the boundary of the disc at $y$ is,
 \begin{equation*}
 \beta(t)=(1-s(t))p+s(t) y,
 \end{equation*}
 where $\displaystyle s(t)=\frac{e^t-e^{-t}}{e^t+k(p,y) e^{-t}}$.
 \end{proof}
 
\begin{rem}
In the next section, we calculate the forward Busemann function for the Funk disc and the Busemann function for the Hilbert disc explicitly. For this purpose, the form of the Funk and the Hilbert geodesics obtained  in terms of the initial and the hitting boundary point turns out to be more useful.
\end{rem}

\section{The Busemann function and the horocycles in the Funk and the Hilbert disc $\mathbb{D}$}
 The Busemann function for a geodesic $\gamma$ in a Riemannian or in a Finsler manifold can be interpreted as the distance function from ``$\gamma(\infty)$". The Busemann functions play a vital role in  
 studying the geometry of the underlying manifold. See eg., \cite{PDA}, \cite{MK}, \cite{RS}, \cite{KSH}, \cite{KSHI}, \cite{HIT},   for the insightful discussions about the Busemann functions in both complete Riemannian and Finslerian manifolds. \\\\
 In this section,  we first calculate the forward Busemann function in the Funk disc, and then in the Hilbert  disc.
We first recall the definition of the forward Busemann function in the context of Finsler geometry. To define the
Busemann function we need {\it line} and {\it ray} in 
a Finsler manifold. For details see Section $3$ of \cite{SOH}. In the sequel, $(M,F)$ denotes a forward complete, non-compact Finsler manifold without boundary.
\begin{defn}
[\textbf{Forward Ray}]
 A unit speed forward geodesic 
 $\gamma :[0,\infty) \rightarrow M$ with $\gamma(0)=p $, 
 $\dot{ \gamma}(0)=v $ is called a \textit{forward ray}, if 
 $d_F(\gamma(s),\gamma(t))=t-s$, for all $s,t\in [0,\infty)$ with $s < t$. Thus, $\gamma$ is a globally minimizing forward geodesic.
\end{defn}

 \noindent
 Similarly, one can analogously define a {\it backward ray}. 
 \begin{defn}
  [\textbf{Line}]  
 A unit speed geodesic $\gamma :(-\infty,\infty)  \rightarrow M$ with $\gamma(0)=p$, $\dot{ \gamma}(0)=v$, is a \textit{line}, if   $d_F(\gamma(s),\gamma(t))=t-s,$ for all $s,t\in (-\infty,\infty)$ with $ s < t$.
Thus, $\gamma$ is a globally minimizing  geodesic.
 \end{defn}

\vspace{.3cm}
\noindent
Associated to a forward ray $\gamma$, we consider the generalized distance function on  $(M, F)$ given by,
$b_{\gamma,t} : M \rightarrow \mathbb{R}$,
\begin{equation}
b_{\gamma,t}(x)=t-d_F(x, \gamma(t)),
\end{equation}
where $d_F$ is the Finsler asymmetric distance. It follows from the triangle inequality that, $b_ {\gamma,t} (x)$  is monotonically increasing function with $t$ and $-d_F(x,\gamma(0)) < b_\gamma(x) < d_F(x, \gamma(0))$.
\begin{defn}
[\textbf{The forward  Busemann function}] The \textit{forward  Busemann function} for the forward geodesic ray $\gamma$ is defined as:
\begin{equation}{\label{eqn 3.6.0}}
  b_\gamma (x):= \lim\limits_{t \rightarrow \infty}\left\lbrace t-d_F(x,\gamma(t) )\right\rbrace.
\end{equation}
By the above discussion, the limit exist and therefore, $b_\gamma$ is well defined.
\end{defn}

\begin{rem}
Note that in contrast to the Riemannian case, we   \textit{can't} define the Busemann function in a forward complete Finsler manifold as:
\begin{equation}\label{RD}
b_\gamma(x):=\lim\limits_{t \rightarrow \infty}b_\gamma(t)=\lim\limits_{t \rightarrow \infty}\left\lbrace d_F(x,\gamma(t))-t\right\rbrace.
\end{equation} 

\noindent
See Chapter 9, $\S 3.4$ of \cite{PP}, for the detailed definition of Busemann function in Riemannian manifold. In general, for simply connected, complete Riemannian manifold without conjugate points, the Busemann functions are distance functions, that is  $|\nabla b_\gamma| \equiv 1$ (see Proposition $1$ of \cite{JHE}).
The definition (\ref{RD}) leads to $b_\gamma(\gamma(t))=-t$ and in turn ${\nabla} b_\gamma(\gamma(t))=-\dot{ \gamma}(t)$. \\
In case of Finsler manifolds, we expect the Busemann functions to be a distance function (Definition \ref{def 3.A30}). 
And  therefore, $\overrightarrow{\nabla}b_\gamma(\gamma(t))= -\dot{ \gamma}(t)$  may not be a unit vector, as desired unless the space is reversible. Hence, for Finsler manifolds it is appropriate to define the Busemann function by (\ref{eqn 3.6.0}).
\end{rem}

 \subsection{The forward Busemann function on the Funk disc}
 \begin{thm}
  Let $\gamma$ be the forward unit speed geodesic in $(\mathbb{D}, F_F)$ with initial point $p \in \mathbb{D}$  and the forward hitting point at $y\in \partial \mathbb{D}$. Then the forward Busemann function for $\gamma$ on $(\mathbb{D}, F_F)$ is given by, 
  \begin{equation}\label{eqn 3.A46}
  b_\gamma(x)=\ln\frac{1-\langle p , y\rangle }{1-\langle x , y\rangle}.
  \end{equation}
 \end{thm}
\begin{proof}
 Let $\gamma(t)=e^{-t} p + (1-e^{-t})y$ be a forward unit speed geodesic in $(\mathbb{D}, F_F)$  given by  \eqref{fgeo}. Then for $x \in \mathbb{D}$ and  $a = \overrightarrow{x\gamma(t)} \cap \partial \mathbb{D}$,
  \begin{equation}{\label{3.A01}}
  d_F(x,\gamma(t))=\ln\frac{|x-a|}{|\gamma(t)-a|}.
  \end{equation}
  The line passing through $x=(x^1,x^2)$ and $\gamma(t)=(\gamma_1(t),\gamma_2(t)) $ is given by,
      \begin{equation*}
     \frac{X^1-\gamma_1(t)}{x^1-\gamma_1(t)}=\frac{X^2-\gamma_2(t)}{x^2-\gamma_2(t)}=\lambda(x, \gamma(t)), 
     \end{equation*}
      where $\lambda (x, \gamma(t))$ is a continuous parameter. Therefore,
      \begin{equation}
      X^1=\gamma_1(t)+\lambda (x, \gamma(t)) (x^1-\gamma_1(t)) ~~ \mbox{and} ~~ X^2=\gamma_2(t)+\lambda (x, \gamma(t)) (x^2-\gamma_2(t)).
      \end{equation}
     If this line intersects the unit circle, then we get,
     \begin{equation}\label{int}
         \lambda^2(x, \gamma(t)) \Big[|\gamma(t)-x|^2\Big] + 2 \lambda (x, \gamma(t)) \Big[ \langle x , \gamma(t)\rangle  -|\gamma(t)|^2\Big]  -\Big[1-|\gamma(t)|^2 \Big] =0.
     \end{equation}
     
\noindent 
 The roots of the above equations are $\lambda_1(x,\gamma(t))$ and $\lambda_2(x,\gamma(t))$,  given by
      \begin{equation}\label{eqn 3.A15}
     \lambda_1(x, \gamma(t))  =\frac{|\gamma(t)|^2-\langle x , \gamma(t)\rangle-\sqrt{[|\gamma(t)|^2-\langle x , \gamma(t)\rangle ]^2+|\gamma(t)-x|^2(1-|\gamma(t)|^2) }}{|\gamma(t)-x|^2},
      \end{equation}
  \begin{equation}\label{eqn 3.A16}
  \lambda_2(x, \gamma(t)) =\frac{|\gamma(t)|^2-\langle x , \gamma(t)\rangle+\sqrt{[|\gamma(t)|^2-\langle x , \gamma(t)\rangle ]^2+|\gamma(t)-x|^2(1-|\gamma(t)|^2) }}{|\gamma(t)-x|^2}.
  \end{equation}
  \noindent
 By \eqref{eqn 3.A15}, we observe that  for all $x, \gamma(t) \in \mathbb{D}$, $\lambda_1(x, \gamma(t)) < 0$ and  by \eqref{int} $\lambda_1 (x, \gamma(t)) \cdot \lambda_2(x, \gamma(t)) < 0$;
 consequently, $\lambda_2(x, \gamma(t))  > 0$.   Also, $\lambda_1(x, \gamma(t)) \rightarrow 0~ \mbox{as}~ t \rightarrow \infty $   since $\gamma(t) \rightarrow y~ \mbox{as}~ t \rightarrow \infty $.
 Let  $a=\overrightarrow{x\gamma(t)}\cap \partial \mathbb{D}$.\\
  Therefore,
      \begin{equation}\label{eqn 3.A00}
      |\gamma(t)-a|^2 =\lambda^2_1(x, \gamma(t))|x-\gamma(t)|^2 ~~ \mbox{and}~~   |x-a|^2  =(1-\lambda_1(x, \gamma(t)))^2|x-\gamma(t)|^2.
      \end{equation}
     
      \noindent
       Substituting \eqref{eqn 3.A00}  in \eqref{3.A01} we get, 
      \begin{equation}\label{3.A12}
      d_F(x,\gamma(t))=\ln\frac{1-\lambda_1(x, \gamma(t))}{|\lambda_1(x, \gamma(t))|}.
      \end{equation}
Now again using \eqref{eqn 3.A15} in \eqref{3.A12}  yields,
\begin{eqnarray}\label{eqn 3.23}
b_\gamma(x)=\lim\limits_{t \rightarrow \infty}\left\lbrace t-d_F(x,\gamma(t) )\right\rbrace = \lim\limits_{t \rightarrow \infty}\ln\left\lbrace e^{t}\cdot \frac{|\lambda_1(x, \gamma(t))|}{1-\lambda_1(x, \gamma(t))}\right\rbrace.
\end{eqnarray}
Finally, using \eqref{eqn 3.A15} in \eqref{eqn 3.23} and 
then, substituting $\gamma(t)= e^{-t} p + (1-e^{-t})y$; after some simplifications we obtain,  
\begin{equation*}
b_\gamma(x)=\ln\frac{1-\langle p , y\rangle }{1-\langle x , y\rangle}.
\end{equation*}
\end{proof}

\subsection{The Busemann function on the Hilbert disc}
\begin{thm}
Let  $\beta$ be the  unit speed geodesic line  in  the Hilbert  disc $(\mathbb{D}, F_H)$ with initial point $p\in \mathbb{D}$ and the forward hitting point  $y\in \partial \mathbb{D}$. Then the Busemann function for the line $\beta$ is, $$b_\beta(x)=\frac{1}{2}\ln\frac{(1-|x|^2)(1-\langle p , y\rangle)^2}{(1-|p|^2)(1-\langle x , y\rangle)^2}.$$   
\end{thm}
\begin{proof}
Let  $\beta$ be the  unit speed geodesic line in  the Hilbert disc $(\mathbb{D}, F_H)$  given by \eqref{eqn 3.4.21}. Then for  $x \in  \mathbb{D}$, we have
\begin{equation}{\label{eqn 3.51}}
b_\beta(x)=\lim\limits_{t \rightarrow \infty}\left\lbrace t-d_H(x,\beta(t) )\right\rbrace.
\end{equation}
And
\begin{equation}{\label{eqn 3.A53}}
 d_H(x,\beta(t) )=\frac{1}{2}\ln\frac{|x-a|\cdot |\beta(t)-b|}{|\beta(t)-a|\cdot |x-b|},
 \end{equation}
 where $a=\overrightarrow{x\gamma(t)}\cap \partial \mathbb{D} $ and  $b=\overleftarrow{x\gamma(t)}\cap \partial  \mathbb{D} $.\\
 Then as in the Funk case, by considering the equation of the line passing through $x=(x^1,x^2)$ and $\beta(t)=(\beta_1(t),\beta_2(t))$; we obtain that, if this line intersects the unit circle, then the following equation holds.
 \begin{equation}\label{aa}
      \lambda^2(x, \beta(t)) \Big[|\beta(t)-x|^2\Big] +2\lambda (x, \beta(t)) \Big[ \langle x , \beta(t)\rangle  -|\beta(t)|^2\Big] -\Big[ 1-|\beta(t)|^2\Big] =0.
 \end{equation}
\noindent
     Let the roots of the above equations be $\lambda_1(x, \beta(t))$ and $\lambda_2(x, \beta(t))$, then 
\begin{equation}\label{eqn 3.A17}
\lambda_1(x, \beta(t)) =\frac{|\beta(t)|^2-\langle x , \beta(t)\rangle-\sqrt{[|\beta(t)|^2-\langle x , \beta(t)\rangle ]^2+|\beta(t)-x|^2(1-|\beta(t)|^2) }}{|\beta(t)-x|^2},
\end{equation}
\begin{equation}\label{eqn 3.A18}
\lambda_2(x, \beta(t)) =\frac{|\beta(t)|^2-\langle x , \beta(t)\rangle+\sqrt{[|\beta(t)|^2-\langle x , \beta(t)\rangle ]^2+|\beta(t)-x|^2(1-|\beta(t)|^2) }}{|\beta(t)-x|^2}.
\end{equation}
 Clearly by \eqref{eqn 3.A17}, $\lambda_1(x, \beta(t)) < 0$ for all 
 $x, \beta(t) \in \mathbb{D}$ and by \eqref{aa}, $\lambda_1(x, \beta(t)) \cdot \lambda_2(x, \beta(t)) < 0$. Therefore, 
 $\lambda_2(x, \beta(t)) > 0$ for all $x, \beta(t) \in \mathbb{D}$. Also, $\lambda_1(x, \beta(t)) \rightarrow 0~ \mbox{as}~ t \rightarrow \infty $, since $\beta(t) \rightarrow y~ \mbox{as}~ t \rightarrow \infty $.\\
Again following the calculation as in the Funk case we obtain,
\begin{equation}\label{eqn 3.A20}
|x-a|^2 =(1-\lambda_1^2(x, \beta(t)))|x-\beta(t)|^2,~|x-b|^2  =(1-\lambda_2^2(x, \beta(t)))|x-\beta(t)|^2,
\end{equation}
\begin{equation}\label{eqn 3.A22}
|\beta(t)-a|^2 =\lambda_1^2(x, \beta(t))|x-\beta(t)|^2,~|\beta(t)-b|^2 =\lambda_2^2(x, \beta(t))(x, \beta(t))|x-\beta(t)|^2.
\end{equation}
\noindent
Also it is clear that from \eqref{eqn 3.A20} that $\lambda_2(x, \beta(t)) \neq 1 (\mbox{as} \;x \neq  b) $.\\
    Substituting  \eqref{eqn 3.A20},\eqref{eqn 3.A22} in \eqref{eqn 3.A53} yields,
     \begin{equation}\label{eqn 3.A54}
     d_H(x,\beta(t) )=\frac{1}{2}\ln\frac{\lambda_2(x, \beta(t)))\cdot (1-\lambda_1(x, \beta(t))))}{|\lambda_1(x, \beta(t)))|\cdot |1-\lambda_2(x, \beta(t)))|}.
     \end{equation}
  Therefore, from \eqref{eqn 3.51} and \eqref{eqn 3.A54}  we get,
     \begin{equation}\label{eqn 3.A13}
      b_\beta(x)=\frac{1}{2}\lim\limits_{t \rightarrow \infty}\ln\left\lbrace e^{2t}\cdot\frac{|\lambda_1(x, \beta(t)))| \cdot |1-\lambda_2(x, \beta(t)))|}{\lambda_2(x, \beta(t))) \cdot (1-\lambda_1(x, \beta(t))))}\right\rbrace.
     \end{equation}
Substituting $\beta(t)=(1-s(t))p+s(t) y$ (see \eqref{eqn 3.4.21}) in \eqref{eqn 3.A13}, after some simplifications we obtain,
 \begin{equation*}
 b_\beta(x)=\frac{1}{2}\ln\frac{(1-|x|^2)(1-\langle p , y\rangle)^2}{(1-|p|^2)(1-\langle x , y\rangle)^2}.
 \end{equation*}
 \end{proof}
 
 \subsection{Horocycles in the Funk and the Hilbert disc}
 Let $(M,F)$ be a forward complete, simply connected Finsler manifold without conjugate points.  Then for $q \in M$, the forward and the backward spheres are, respectively, defined by (\cite{DSSZ}, $\S 6.2~B$),
 \begin{equation*}
  S(q,r)^+=\left\lbrace x \in M | d_F(q, x)=r \right\rbrace ~~ \mbox{and}~~   S(q,r)^-=\left\lbrace x \in M | d_F(x,q)=r \right\rbrace.  
 \end{equation*}
 \noindent
 \begin{defn}
 [Forward Horosphere]
 Let $\gamma :[0,\infty) \rightarrow M$ be a forward ray with $\gamma(0)=p$, $\dot{\gamma}(0)=v$. 
 Then,
  \begin{equation}
  b_\gamma^{-1}(a)=\lim\limits_{t \rightarrow \infty}S(\gamma(t),t-a)^-, 
  \end{equation} 
 the limit of the backward spheres is called the {\it forward horosphere} passing through $p\in M$. In dimension two horospheres are termed as {\it horocycles}.
 \end{defn}
  
  \vspace{.3cm}
  \noindent
  We describe the forward  horocycles of the Funk disc explicitly.  Note that in the Funk disc backward horocycles do not exist.
  \begin{proposition}
 Let $\gamma(t)=e^{-t} p+ (1-e^{-t})y$ be the Funk forward geodesic of~ $\mathbb{D}$ starting with $p$ and the forward hitting point at $y \in \partial \mathbb{D}$. Then the forward horocycles along the $\gamma$ are line segments in the Funk disc perpendicular to the line joining origin and the forward hitting point at $y$. 
  \end{proposition}
 
 \begin{proof}
 In the case of Funk disc, the forward Busemann function for the unit speed forward geodesic  $\gamma$ is given by \eqref{eqn 3.A46}: $$b_\gamma(x)=\ln\frac{1-\langle p , y\rangle }{1-\langle x , y\rangle }.$$ Therefore, it follows that for $a \in \mathbb{R}$,
  \begin{eqnarray}
 \nonumber b_\gamma^{-1}(a) = \left\lbrace (x^1,x^2)\in \mathbb{D}^2 : x^1y^1+x^2y^2=1-e^{-a}(1-\langle p , y\rangle) \right\rbrace. 
  \end{eqnarray}
 Clearly, the line  $x^1y^1+x^2y^2=1-e^{-a}(1-\langle p , y\rangle)$ is perpendicular to the line $x^1y^1-x^2y^2=0$.
 
 \end{proof}

\vspace{-0.2in}
\begin{proposition}
    The forward and the  backward horocycles of the Hilbert disc are the ellipses. 
\end{proposition}
 \begin{proof}
 In  case of the Hilbert disc, the Busemann function for the unit speed line $\beta$  given by \eqref{eqn 3.4.21} is,
  $$ b_\beta(x)=\frac{1}{2}\ln\frac{(1-|x|^2)(1-\langle p , y\rangle)^2}{(1-|p|^2)(1-\langle x , y\rangle)^2}.$$
  Therefore, it follows that for $a \in \mathbb{R}$,
 \begin{eqnarray}\label{eqn 3.A48}
  b_\beta^{-1}(a)= \left\lbrace (x^1,x^2)\in \mathbb{D}^2 : \frac{1}{2}\ln\frac{(1-|x|^2)(1-\langle p , y\rangle)^2}{(1-|p|^2)(1-\langle x , y\rangle)^2}=a\right\rbrace.
 \end{eqnarray}
 Simplifying  RHS of \eqref{eqn 3.A48}, we obtain  the following  \eqref{eqn 3.A47} which represents the ellipse.
 
 \begin{equation}\label{eqn 3.A47} 
 \begin{split}
 (x^1)^2\Big[e^{2a}(1-|p|^2)(y^1)^2+(1-\langle p , y\rangle)^2\Big]+(x^2)^2\Big[e^{2a}(1-|p|^2)(y^2)^2+(1-\langle p , y\rangle)^2\Big]\\+2x^1x^2y^1y^2e^{2a}(1-|p|^2)-2x^1y^1(1-|p|^2)-2x^2y^2(1-|p|^2)\\+\Big[e^{2a}(1-|p|^2)-(1-\langle p , y\rangle)^2\Big]=0.
 \end{split}
 \end{equation} 
 \end{proof}
 \section{The Forward Asymptotic Harmonicity  of the Funk disc}
 The concept of {\it asymptotically harmonic} Riemannian manifolds  was originally introduced by Ledrappier [\cite{LF}, Theorem 1] in connection with the rigidity of measures related to the Dirichlet problem (harmonic measure) and the dynamics of the geodesic flow (Bowen-Margulis measure). The concept of the asymptotic harmonicity of the Finsler manifold was extended in Definition 4.3  of \cite{HIT}. 
 Recall that the Hessian of a smooth function $f$ is well defined only on the set $ {\cal{U}}_f$; where ${\cal{U}}_f:=  \left\lbrace x\in M : df_x \neq 0 \right\rbrace$ (Definition \ref{def 3.A1}). Using this we {\it correct} the Definition 4.3 \cite{HIT} as follows.
                                                         
\begin{defn}
 [\textbf{Forward Asymptotically Harmonic Finsler Manifold}] Let $(M,F,d\mu)$ be a forward complete, simply connected Finsler $\mu$-manifold without conjugate points. Then $M$ is said to be forward  asymptotically harmonic, if there exist a constant $h$, independent of $x\in M$ and $\gamma$, such that  $\overrightarrow{\Delta}_\mu b_\gamma(x) \equiv h$ for $x\in {\cal{U}}_{b_\gamma}$, in the sense of distributions.
 \end{defn}
 
 \noindent
 In this section, we show that the Funk BH-disc $(\mathbb{D},F_{F}, BH)$ is forward asymptotically harmonic Finsler surface. Recall that the Hilbert metric on the unit disc is the well-known Riemannian Beltrami Klein metric and is known to be asymptotically harmonic \cite{RS}.
 We also show that the Funk-HT disc $(\mathbb{D},F_{F}, HT)$, the Funk-max disc $(\mathbb{D},F_{F}, \max)$, and the Funk-min disc $(\mathbb{D},F_{F}, \min)$ are {\it not} forward asymptotically harmonic Finsler surfaces.
 Towards this, we  show by the detailed calculations that 
 $\overrightarrow{\Delta}_{BH} b_\gamma(x) = -2$, for $x\in M$, where $\overrightarrow{\Delta}_{BH}$ denotes the forward Shen's Laplacian on $M$. 
 But on the other hand, $\overrightarrow{\Delta}_{HT}b_\gamma(x)$, $\overrightarrow{\Delta}_{\max} b_\gamma(x)$, and
$\overrightarrow{\Delta}_{\min}b_\gamma(x)$ are \textit{not} constants for $x \in M$.
 
\subsection{The Dual of the Funk metric on disc $\mathbb{D}$}
In this subsection, we find the   dual  (See Definition \ref{def 3.A1}) of the Funk metric on the unit disc $\mathbb{D}$. This result is important on its own right.
This dual metric is required to calculate the  Laplacian of the Busemann function.

\begin{proposition}\label{ppn 3.A1} The dual $F^*$ of the Funk metric  $F=\alpha+\beta$, on the unit disc $\mathbb{D}$  
 defined  by \eqref{eqn 3.A1} is $F^*=\alpha^*+\beta^*$, where $\alpha^*(x,\xi)=|\xi|$ and $\beta^*(x,\xi)=- \langle x ,  \xi \rangle$, for all $ x\in \mathbb{D}, \xi\in T^*_x\mathbb{D}$. Consequently, $F^*$ is a  Randers metric.
\end{proposition}
\begin{proof}
 The Funk metric on the unit disc $\mathbb{D}$,  given by \eqref{eqn 3.A1}, can be  rewritten as
  \begin{eqnarray}
  \nonumber \label{eqn 3.A24} F_F(x,v)&=&\sqrt{a_{ij}(x)v^iv^j}+b_i(x)v^j, 
   \end{eqnarray}
  where 
  $ \displaystyle a_{ij}(x)=\frac{\delta_{ij}(1-|x|^2)+x_ix_j}{(1-|x|^2)^2}$,
  $\displaystyle b_i(x)=\frac{\delta_{ij} x^j}{1-|x|^2}$ and $\displaystyle x_i=\delta_{ij}x^j$. 
  The components  $ a^{ij}(x)$ of the inverse of the matrix $(a_{ij}(x))$ are given by,
  \begin{equation}\label{eqn 3.A28}
  a^{ij}(x)=(1-|x|^2)(\delta^{ij}-x^ix^j).
  \end{equation}
  Clearly,
    \begin{equation}\label{eqn 3.A30}
  ||\beta||^2_\alpha=a^{ij}(x)b_i(x)b_j(x)=|x|^2.
  \end{equation}
  Using the techniques of Example 3.1.1  in \cite{SZ}, we find $F^*$ explicitly as follows.
  \begin{equation}\label{eqn 3.A31}
   F^*(x, \xi)=\alpha^*(x, \xi)+\beta^*(x, \xi) =\sqrt{a^{*ij}(x)\xi_i\xi_j}+b^{*i}(x)\xi_i,
  \end{equation}
  where $\xi=(\xi_i) \in T_x^*\mathbb{D}$, 
  $\displaystyle a^{*ij}(x)=\frac{(1-|\beta|^2_\alpha)a^{ij}(x)+b^i(x)b^j(x)}{(1-||\beta||^2_\alpha)^2}=\delta^{ij},$ \\and
  $\displaystyle b^{*i}(x)=-\frac{b^i(x)}{1-||\beta||^2_\alpha}=-x^i,$
  where $b^i(x)=a^{ij}(x)b_j(x)=x^i(1-|x|^2)$.\\
  Therefore,
  \begin{equation}\label{eqn 3.A25}
  F^*(x, \xi)=|\xi|-\langle x , \xi \rangle.
  \end{equation}
  Further, since $||\beta^*||_{\alpha^*}=|x| < 1$. Hence $F^*$ is a  Randers metric.
  \end{proof}
  
 \vspace{0.3cm}
  
 \begin{corollary}
 The Funk-Busemann functions are distance functions (see Definition \ref{def 3.A30}) on the Funk disc and  consequently,  
 ${\cal{U}}_{b_\gamma}= \mathbb{D}$.
 \end{corollary}
 \begin{proof}
 As $\displaystyle b_\gamma(x)=\ln\frac{1-\langle p , y\rangle }{1-\langle x , y\rangle }$, $\displaystyle db_{\gamma_ {\mid x}}=\frac{y^1 dx^1 +y^2 dx^2}{1-\langle x , y\rangle} \neq 0 $.
Hence by \eqref{eqn 3.A25}, $F^*(x, db_{\gamma_ {\mid x}})=1 = F(x, \nabla b_\gamma(x))$. Consequently, the
Funk-Busemann functions are {\it distance functions}.
\end{proof}
 
\subsection{The  forward Shen Laplacian of the  forward Busemann function in the Funk disc} 
In this subsection, we compute the forward Shen Laplacian of the forward Busemann function in the Funk disc  directly by the formula (\ref{eqn 3.A52}). Let $\gamma(t) = pe^{-t}+(1-e^{-t})y$  be the forward geodesic
in the Funk disc $\mathbb{D}$ given by (\ref{fgeo}) with $\gamma(0)=p$ and $\dot{ \gamma}(0)=y-p$.  

\begin{thm}
    We have for all $x\in \mathbb{D}$, $\overrightarrow{\Delta}_{BH} b_\gamma(x) = -2$.
\end{thm}

\begin{proof} 
Let $(M,F,d\mu)$ be a Finsler manifold. Then by \eqref{eqn 3.A52}  we obtain,
\begin{equation}\label{eqn 3.A37}
\overrightarrow{\Delta} b_\gamma(x) =\frac{1}{\sigma_{BH}(x)}\frac{\partial}{\partial x^i}\left(\sigma_{BH}(x) g^{*ij}(x, db_{\gamma_ {\mid x}})\frac{\partial b_\gamma(x)}{\partial x^j}\right).
\end{equation}
By (1.6) of Example 1.2.1  in \cite{SZ}, we get
\begin{eqnarray}
  \nonumber  g^{*ij}(x,\xi)&=&\frac{1}{2}[F^{*2}(x,\xi)]_{\xi^i\xi^j}\\
  \nonumber  &=& \frac{F^*(x,\xi)}{\alpha^*(x,\xi)}\left(a^{*ij}(x)-\frac{\xi^i}{\alpha^*(x,\xi)} \frac{\xi^j}{\alpha^*(x,\xi)}\right)\\  \nonumber&&+ \left( \frac{\xi^i}{\alpha^*(x,\xi)}-x^i\right) \left(\frac{\xi^j}{\alpha^*(x,\xi)} -x^j \right), 
\end{eqnarray}
where $\xi^i=a^{*is}(x)\xi_s=\delta ^{is}\xi_s=\xi_i$.\\
Substituting  $\alpha^*(x, \xi)=|\xi|$ and $F^*(x, \xi)=|\xi|-\langle x , \xi \rangle$ (Proposition \ref{ppn 3.A1}) in the above equation  yields,
\begin{equation}\label{eqn 3.A33}
g^{*ij}(x, \xi)=\left(1-\frac{\langle x , \xi \rangle}{|\xi|} \right) \left(\delta^{ij}-\frac{\xi_i}{|\xi|} \frac{\xi_j}{|\xi|}\right) +\left( \frac{\xi_i}{|\xi|}-x^i\right) \left(\frac{\xi_j}{|\xi|} -x^j \right). 
\end{equation}
Let $\xi=db_{\gamma_ {\mid x}}$ . Then
\begin{equation}\label{eqn 3.A34}
\xi=(\xi_i)=\left( \frac{\partial b_\gamma(x)}{\partial x^i} \right)=\left( \frac{y^i}{1-\langle x , y \rangle}\right)~~ \mbox{and}~~~ |\xi|=\frac{1}{1-\langle x , y \rangle}.
\end{equation}
Also,
\begin{equation}\label{eqn 3.A43}
\frac{\partial^2 b_\gamma(x)}{\partial x^i\partial x^j}=\frac{y^iy^j}{(1-\langle x , y \rangle)^2}.
\end{equation}
Therefore, \eqref{eqn 3.A33} and  \eqref{eqn 3.A34}
yields,
\begin{equation}\label{eqn 3.A35}
g^{*ij}(x, db_{\gamma_ {\mid x}})=\left(1-\langle x , y \rangle \right) \left(\delta^{ij}-y^iy^j\right) +\left( y^i-x^i\right) \left(y^j -x^j \right),
\end{equation}
\begin{equation}\label{eqn 3.A36}
\frac{\partial g^{*ij}(x, db_{\gamma_ {\mid x}})}{\partial x^k}=-\Big[\delta_{lk} y^l \left(\delta^{ij}-y^iy^j\right) +\delta_{jk}\left( y^i-x^i\right)+\delta_{ik}\left(y^j -x^j \right) \Big].
\end{equation}
We also have $\sigma_{BH} (x)=1$ (cf. Lemma \ref{lem 3.A1}). Using \eqref{eqn 3.A34}, \eqref{eqn 3.A43},  \eqref{eqn 3.A35}, \eqref{eqn 3.A36} in \eqref{eqn 3.A37}, after some simplifications we get,
\begin{eqnarray}
\label{eqn 3.A42} \overrightarrow{\Delta}_{BH} b_\gamma(x)&=&\frac{\partial}{\partial x^i}\left(  g^{*ij}(x, db_{\gamma_ {\mid x}})\frac{\partial b_\gamma(x)}{\partial x^j}\right)\\ \nonumber &=&\frac{\partial}{\partial x^i}\left(g^{*ij}(x, db_{\gamma_ {\mid x}})\right)\frac{\partial b_\gamma(x)}{\partial x^j}+g^{*ij}(x, db_{\gamma_ {\mid x}})\left( \frac{\partial^2 b_\gamma(x)}{\partial x^i\partial x^j}\right)\\ \nonumber &=&-\Big[ y^i \left(\delta^{ij}-y^iy^j\right) +\delta_{ji}\left( y^i-x^i\right)+\left(y^j -x^j \right) \Big]\left( \frac{y^j}{1-\langle x , y \rangle}\right)\\ \nonumber && +\Big[ \left(1-\langle x , y \rangle \right) \left(\delta^{ij}-y^iy^j\right) +\left( y^i-x^i\right) \left(y^j -x^j \right)\Big]\left( \frac{y^iy^j}{(1-\langle x , y \rangle)^2}\right)\\ \nonumber &=&-2.
\end{eqnarray}
\end{proof}
\begin{corollary}\label{cor 3.A1}
We have ,
\begin{itemize}
\item[(1)] $\displaystyle \overrightarrow{\Delta}_{HT} b_\gamma(x) = \overrightarrow{\Delta}_{BH} b_\gamma(x) +\frac{3(\langle x , y \rangle-|x|^2)}{1-|x|^2}$, ~~ $\forall x\in \mathbb{D}$.
\item[(2)] $\displaystyle \overrightarrow{\Delta}_{max} b_\gamma(x) = \overrightarrow{\Delta}_{BH} b_\gamma(x) +\frac{3(\langle x , y \rangle-|x|^2)}{|x|(1-|x|^2)}$,  ~~ $\forall x\in \mathbb{D} \setminus \left\lbrace 0 \right\rbrace$.
\item[(3)] $ \displaystyle \overrightarrow{\Delta}_{min} b_\gamma(x) = \overrightarrow{\Delta}_{BH} b_\gamma(x) -\frac{3(\langle x , y \rangle-|x|^2)}{|x|(1-|x|^2)}$, ~~ $\forall x\in \mathbb{D} \setminus \left\lbrace 0 \right\rbrace$.
\end{itemize}
\end{corollary}
\begin{proof} (1)
For $\displaystyle \sigma_{HT} (x)=\frac{1}{(1-|x|^2)^{\frac{3}{2}}}$ (cf. Lemma \ref{lem 3.A1}). Therefore, \eqref{eqn 3.A37} implies that,
\begin{eqnarray}
\nonumber \overrightarrow{\Delta}_{HT} b_\gamma(x)&=&\overrightarrow{\Delta}_{BH} b_\gamma(x) +\left\lbrace \left(1-\langle x , y \rangle \right) \left(\delta^{ij}-y^iy^j\right) +\left( y^i-x^i\right) \left(y^j -x^j \right)\right\rbrace \\ \nonumber  && \times \frac{y^j}{1-\langle x , y \rangle}\times \frac{3x^i}{1-|x|^2},\\
\nonumber&=&\overrightarrow{\Delta}_{BH} b_\gamma(x)+\frac{3(\langle x , y \rangle-|x|^2)}{1-|x|^2}.
\end{eqnarray}
(2)
For $\displaystyle \sigma_{max} (x)=\frac{(1+|x|)^{\frac{3}{2}}}{(1-|x|)^{\frac{3}{2}}}$ (cf. Lemma \ref{lem 3.A1}). Hence, \eqref{eqn 3.A37} gives,
\begin{eqnarray}
\nonumber \overrightarrow{\Delta}_{max}b_\gamma(x)&=&\overrightarrow{\Delta}_{BH} b_\gamma(x)+\left\lbrace \left(1-\langle x , y \rangle \right) \left(\delta^{ij}-y^iy^j\right) +\left( y^i-x^i\right) \left(y^j -x^j \right)\right\rbrace  \\ \nonumber  && \times \frac{y^j}{1-\langle x , y \rangle} \times  \frac{3x^i}{|x|(1-|x|^2)},\\
\nonumber&=&\overrightarrow{\Delta}_{BH} b_\gamma(x)+\frac{3(\langle x , y \rangle-|x|^2)}{|x|(1-|x|^2)}.
\end{eqnarray}
(3)
For $\displaystyle \sigma_{min} (x)=\frac{(1-|x|)^{\frac{3}{2}}}{(1+|x|)^{\frac{3}{2}}}$ (cf. Lemma \ref{lem 3.A1}). Consequently, using \eqref{eqn 3.A37} we obtain,
\begin{eqnarray}
\nonumber \overrightarrow{\Delta}_{max} b_\gamma(x)&=&\overrightarrow{\Delta}_{BH}b_\gamma(x)+\left\lbrace \left(1-\langle x , y \rangle \right) \left(\delta^{ij}-y^iy^j\right) +\left( y^i-x^i\right) \left(y^j -x^j \right)\right\rbrace  \\ \nonumber  && \times \frac{y^j}{1-\langle x , y \rangle} \times  \frac{-3x^i}{|x|(1-|x|^2)},\\
\nonumber&=&\overrightarrow{\Delta}_{BH} b_\gamma(x) -\frac{3(\langle x , y \rangle-|x|^2)}{|x|(1-|x|^2)}.
\end{eqnarray}
\end{proof}

\noindent
\begin{rem}
\begin{itemize}
\item[(1)]  Corollary \ref{cor 3.A1} shows that all $\overrightarrow{\Delta}_\mu$  
 ($\mu = \mbox{BH, HT, max, min}$)  are  different. Hence, 
 we see that the concept of asymptotic harmonicity of Finsler manifolds strictly depends on the measure, in contrast to the Riemannian case.
 \item[(2)] As the Funk metrics are of constant flag curvature $-\frac{1}{4}$,  the mean curvature $\Pi_{ \nabla r }$ of the backward geodesic sphere of radius $r$,  is given by ((3.7) of \cite{ZWSY}, Lemma 3.2):
\begin{equation}\label{mcr}
    \Pi_{\nabla r}=-\frac{(n-1)}{2} \coth \left(\frac{r}{2}\right)-\frac{(n+1)}{2}. 
\end{equation}
So for $n=2$ we obtain, the mean curvature of forward horospheres as,
$$\Pi_{\infty}=-2 = \overrightarrow{\Delta}_{BH} b_\gamma(x).$$

This calculation matches with our direct calculation of Laplacian.
\end{itemize}
\end{rem}

\noindent
From Lemma \ref{lm A 2.10} and \eqref{mcr}  we conclude:
\begin{corollary}
The Finslerian and the induced Riemannian mean curvature, 
respectively, of all the forward horocycles of the Funk disc is constant  $-2$ and $\frac{-1}{2}$, respectively.
\end{corollary}

\noindent
From (\ref{mcr}) we also conclude that:

\begin{corollary}
For the Busemann function on the Hilbert disc  $\Delta b_\gamma(x) \equiv  \frac{-1}{2}$. We recover the well known result that the Hilbert disc
is asymptotically harmonic Riemannian manifold.  
\end{corollary}


\vspace{2cm}
\textbf{Statements and Declarations}\\
\textbf{Funding}\\ The first author is supported by  UGC Senior Research Fellowship, India with Reference No. 1076/(CSIR-UGC NET JUNE 2017). This work was done during the visit of first author at HRI and therefore he gratefully acknowledges
the facilities provided by HRI while this work was completed. 

\vspace{1cm}
\noindent
\textbf{Declarations}\\
\textbf{Conflict of interest} On behalf of all authors, the corresponding author states that there is no conflict of financial or non-financial interests that are directly or indirectly related to the work submitted for publication.

\vspace{1cm}
\noindent
\textbf{Data Availibility} 
This manuscript has no associated data.
\end{document}